\documentclass[12pt]{amsart}

\usepackage{amsmath,amstext,amssymb,amsopn,amsthm}
\usepackage{url,verbatim}
\usepackage{mathtools}
\usepackage{enumerate}

\usepackage[f]{esvect}

\usepackage{color,graphicx}

\usepackage[margin=30mm]{geometry}
\usepackage{eucal,mathrsfs,dsfont}

\allowdisplaybreaks

\newtheorem{theorem}{Theorem}[section]

\newtheorem{lemma}[theorem]{Lemma}

\theoremstyle{definition}

\numberwithin{equation}{section}

\newcommand{\eps}{\varepsilon}

\newcommand{\calL}{\mathcal{L}}
\newcommand{\calF}{\mathcal{F}}

\newcommand{\calD}{\mathcal{D}}

\newcommand{\calT}{\mathcal{T}}
\newcommand{\calS}{\mathcal{S}}

\newcommand{\calC}{\mathcal{C}}
\newcommand{\calE}{\mathcal{E}}
\newcommand{\calB}{\mathcal{B}}

\newcommand{\calX}{\mathcal{X}}
\newcommand{\calY}{\mathcal{Y}}

\newcommand{\calU}{\mathcal{U}}

\newcommand{\fX}{\mathbf{X}}
\newcommand{\fY}{\mathbf{Y}}
\newcommand{\fB}{\mathbf{B}}
\newcommand{\bK}{\mathbf{K}}
\newcommand{\bM}{\mathbf{M}}
\newcommand{\bN}{\mathbf{N}}
\newcommand{\bR}{\mathbf{R}}

\newcommand{\bT}{\mathbf{T}}

\newcommand{\E}{\operatorname{\mathds{E}}} 
\renewcommand{\P}{\operatorname{\mathds{P}}} 
\newcommand{\R}{\mathds{R}}

\newcommand{\Z}{{\mathbb Z}}

\newcommand{\vphi}{\varphi}
\newcommand{\prt}{\partial}
\newcommand{\bL}{\mathbf{L}}
\newcommand{\bS}{\mathbf{S}}
\newcommand{\bA}{\mathbf{A}}
\newcommand{\bDelta}{\mathbf{\Delta}}
\newcommand{\bPi}{\mathbf{\Pi}}

\newcommand{\ol}{\overline}
\newcommand{\wh}{\widehat}
\newcommand{\wt}{\widetilde}

\DeclareMathOperator{\dist}{dist}

\DeclareMathOperator{\Var}{Var}

\def\bone{{\bf 1}}
\def\n{{\bf n}}

\def\be{{\bf e}}

\title{Stirring two grains of sand}
\author{Krzysztof Burdzy}

\address{KB: Department of Mathematics, Box 354350, University of Washington, Seattle, WA 98195}
\email{burdzy@math.washington.edu}

\thanks{Research supported in part by NSF Grant DMS-1206276. }

\begin{document}

\begin{abstract}

Consider two unit balls in a $d$-dimensional flat torus with edge length $r$, for $d\geq 2$. The balls do not move by themselves but they are pushed by a Brownian motion. The balls never intersect---they reflect if they touch. It is proved that the joint distribution of the processes representing the centers of the balls converges to the distribution of two independent Brownian motions when $r\to \infty$, assuming that we use a proper clock and proper scaling. The diffusion coefficient of the limit process depends on the dimension. The positions of the balls are asymptotically independent also in the following  sense. The rescaled stationary distributions of the centers of the balls converge to the product of the stationary (hence uniform) distributions for each ball separately, as $r\to\infty$.
\end{abstract}

\maketitle

\section{Introduction}\label{intro}

The word ``stirring'' in the title of this paper refers to a random change in a system of many bodies that is caused by a single agent that moves continuously and acts locally. This is in contrast to those stochastic flows where different parts of the moving medium are simultaneously ``pushed'' by different (although possibly correlated) random noises.
In everyday life, stirring typically refers to activities such as stirring coffee in a cup or stirring paint in a bucket. In these situations, stirring the medium with a spoon or a stick causes the bulk of the liquid to move (in a circular fashion). Our model is closer to stirring sand in a sandbox with a stick. In this situation, sand grains are displaced locally and there is no overall motion of the bulk of the sand mass. 

Stirring sand in a sandbox provided motivation for this project but our model is a simplification of the reality in (at least) two significant ways. First, we will consider only two ``sand grains'' represented by balls. This seems to be the crucial step in the analysis of the motion of many ``sand grains'' (see the remarks on \cite{BCP} below). 
Second, the stirring agent will be represented by an infinitely small particle performing Brownian motion. One may consider our results as a first step towards a more realistic model.

In our model, the stirring agent, represented by Brownian motion, is not affected by the motion of ``sand grains.'' The two sand grains (balls) remain motionless except when they are pushed by the Brownian particle aside, when its trajectory hits their surfaces.

The problem that we will investigate is that of the evolution of the vector between the centers of the two balls. It is natural to guess that the motion of a each ball should be similar to that of Brownian motion on the local time scale. The crux of the problem is that the directions of the push that the balls receive from the Brownian particle are not independent. Therefore, even if the guess about the motion of a single ball is correct, that does not immediately imply that the limit distribution for the pair of the balls is a pair of independent Brownian motions. 
We will prove that this is in fact true and we will express this idea in two different ways, to be described below.
The main technical challenge of the paper is to estimate the magnitude of the dependence between motions of the two balls. 

We will separately prove the invariance principle for a single ball pushed by Brownian motion in dimension 2 in the whole space $\R^2$. This is meaningful because Brownian motion is recurrent in two dimensions so it will keep pushing the ball forever. We consider this simplified question separately to present a more or less straightforward proof. Many technical details obscure this part of the argument in the case of two balls or in higher dimensions.

 In dimensions 3 and higher, the two balls and Brownian motion will be located in a torus because Brownian motion is transient in these dimensions (but the theorem will cover the two dimensional case as well). First, we will prove an  invariance principle on the local time scale for the centers of the two balls. The limiting process is a pair of two independent Brownian motions. Next we will show that the rescaled stationary distributions for the two balls in a torus of diameter $r$ converge to the product of the stationary (and hence uniform) distributions for the individual balls as $r\to \infty$. 
At the end of Section \ref{main} we will explain why the latter theorem does not immediately follow from the former.

The present paper is a part of a larger project. Our present model is ``almost'' equivalent to the model in which a ball with the center moving as a Brownian motion pushes two point-like particles. The equivalence is not complete because the two balls in our model cannot intersect (by assumption) and hence their centers are always at least two units apart. In the other model, the two point-like particles can come arbitrarily close. Their motion was partly analyzed in \cite{BCP}, where it was proved that the distance between the two particles does not converge to 0 in a three dimensional torus. This is very close to proving recurrence for the two-particle process. The main results of the present article show, more or less, that the particles are independent on the large scale. Only one element of the program initiated in \cite{BCP} is still missing---the positive recurrence (as opposed to the mere recurrence) of the two particle motion. If this gap is filled then this will be, most likely, sufficient to prove Conjectures 1.5 and 1.6 in \cite{BCP}.

The paper is organized as follows. The next section contains the formal description of the model. Section \ref{main} presents the statements of the main results. It is followed by Section \ref{sec:exc} with a review of excursion theory and some results on excursion laws. The motion of a single ball is analyzed in Section \ref{single}. Section \ref{local} is devoted to estimates of local time. Section \ref{hit} gives estimates for hitting distributions. The main theorem on invariance principle is proved in Section \ref{clt}. The theorem on convergence of the stationary distributions is proved in two sections, Sections \ref{sec:irreducibility} and \ref{stationary}, the first of which is devoted to the irreducibility of the process.

\newpage

\section{Preliminaries}\label{prelim}

\subsection{Processes in $\R^d$}\label{prelim1}

First we will consider the case when the Brownian motion and the balls are located in $\R^d$ with $d\geq 2$.
We will consider two moving balls with radii 1 and centers $X_t$ and $Y_t$, resp. Brownian motion will be denoted $B_t$.

Let $\calB(x,r)$ denote the open ball with center $x$ and radius $r$
and let $\calS(x,r) = \prt \calB(x,r)$. The two moving balls will be denoted $\calX_t =\calS(X_t, 1)$ and $\calY_t =\calS(Y_t, 1)$.

For $x\in \calS(y,r)$, let $\n(\calS(y,r), x)$ be the unit outward normal vector to $\calS(y,r)$ at $x$. 

We will now describe the effect of the push of $B_t$ on the trajectory of $X_t$.
 Let us ignore the other ball $\calY_t$ for the moment. We assume that $|B_0 - X_0|\geq 1$, a.s. By the results of \cite{LS}, there exist a continuous process $Z_t$  taking values in $\calB(X_0,1)^c$ and a non-decreasing real valued continuous process (``local time'') $L^X_t$ such that $L^X_0 = 0$, $\int_0^\infty \bone_{\{Z_t \notin \calX_0\}} dL^X_t = 0$, and
\begin{align}\label{m9.2}
Z_t = B_t + \int_0^t \n( \calX_0, Z_s) dL^X_s, \qquad t\geq 0.
\end{align}
At this point, a better name for the process $L^X$ would have been $L^Z$ but $L^X$
was used in anticipation of \eqref{j12.3} below.
The process $Z_t$ is Brownian motion reflected on $\calX_0$. We define $X_t$ by
\begin{align}\label{j12.1}
X_t = X_0 - \int_0^t \n( \calX_0,Z_s) dL^X_s, \qquad t\geq 0.
\end{align}
In this way, the ball $\calX_t$ is pushed by Brownian motion $B_t$. Note that we have $B_t \notin \calB(X_t,1)$ for all $t\geq 0$, $\int_0^\infty \bone_{\{B_t \notin \calX_t\}} dL^X_t = 0$, and
\begin{align}\label{j12.3}
X_t = X_0 - \int_0^t \n( \calX_s,B_s) dL^X_s, \qquad t\geq 0.
\end{align}

Next we will consider the motion of the balls $\calX_t$ and $\calY_t$ only, ignoring $B_t$. Suppose that $X_t$ is a continuous process and $|X_0 - Y_0| \geq 2$. We do not want the balls $\calX_t$ and $\calY_t$ to intersect. We apply the results of \cite{LS} once again. There exist a continuous process $V_t$ taking values in $\calB(Y_0,2)^c$ and a non-decreasing real valued continuous process (``local time'') $L'_t$ such that $L'_0 = 0$, $\int_0^\infty \bone_{\{V_t \notin \calS(Y_0,2)\}} dL'_t = 0$, and
\begin{align*}
V_t = X_t + \int_0^t \n(\calS(Y_0,2), V_s) dL'_s, \qquad t\geq 0.
\end{align*}
We define $Y_t$ by
\begin{align}\label{j12.2}
Y_t = Y_0 - \int_0^t \n( \calS(Y_0,2), V_s) dL'_s, \qquad t\geq 0.
\end{align}
In this way, the ball $\calY_t$ is pushed by the ball $\calX_t$. Note that we have $\calB(X_t,1) \cap \calB(Y_t,1) = \emptyset$ for all $t\geq 0$ and $\int_0^\infty \bone_{\{|X_t - Y_t|>2\}} dL'_t = 0$.

Now we will describe the joint evolution of $B_t, X_t$ and $Y_t$. Suppose that $|B_0 - X_0| \geq 1$, $|B_0 - Y_0| \geq 1$, $|X_0 - Y_0| \geq 2$, and $B_0 \notin  \calX_0 \cap  \calY_0$ (the last condition is satisfied if, for example, $|X_0 - Y_0| > 2$). Assume without loss of generality that $B_t$ hits $\calX_0$ strictly before hitting $\calY_0$. Then  we use \eqref{j12.1} and \eqref{j12.2} to define $X_t$ and $Y_t$ until the first time $T_1 \geq 0$ such that $B_{T_1} \in  \calY_{T_1}$. Suppose that $B_{T_1} \notin  \calX_{T_1}$. At this time we switch the roles of $X_t$ and $Y_t$ in the definitions \eqref{j12.1} and \eqref{j12.2}. In other words, $B_t$ is now pushing the ball $\calY_t$ and the ball $\calY_t$ is pushing the ball $\calX_t$. We define in this way processes $X_t$ and $Y_t$ for $t\geq T_1$ until the first time $T_2\geq T_1$ such that $B_{T_2} \in  \calX_{T_2}$. Suppose that $B_{T_2} \notin  \calY_{T_2}$. We continue in this fashion, i.e., we construct stopping times $T_1, T_2, T_3, \dots$, such that $T_n = \inf\{t\geq T_{n-1} : B_{t} \in  \calY_{t}\}$ for odd $n$ and $T_n = \inf\{t\geq T_{n-1} : B_{t} \in  \calX_{t}\}$ for even $n$. The inductive definition is continued as long as $B_{T_n} \notin  \calX_{T_n}$ for odd $n$ and $B_{T_n} \notin  \calY_{T_n}$ for even $n$. 

\begin{lemma}\label{au13.3}
With probability 1, all stopping times $T_n, n\geq 1$, are well defined and $\lim_{n\to \infty} T_n = \infty$.
\end{lemma}

\begin{proof}
The claim in the lemma may be false for two different reasons.
First, it is possible that for some random time $T_\infty$ and $n<\infty$,
\begin{align}\label{au13.5}
\P(T_\infty = T_n <\infty, B_{T_\infty} \in  \calX_{T_\infty} \cap  \calY_{T_\infty})>0.
\end{align}
Second, it may be that all $T_n$'s are well defined and for some random time $T_\infty$ we have $\P(T_\infty<\infty)>0$ and $T_\infty 
= \lim_{n\to \infty} T_n $, a.s. 

We will analyze the second case first. Consider $\omega$ such that all $T_n$'s are well defined and $T_\infty 
= \lim_{n\to \infty} T_n < \infty$. 
Consider an $\eps\in(0,1)$ and let $t_1 < T_\infty$ be so close to $T_\infty$ that $\sup_{t_1 \leq s,t \leq T_\infty} |B_s - B_t| < \eps$. Suppose that $n_1$ is so large that $T_{2n}> t_1$ for $2n \geq n_1$. Consider any $n$ such that $T_{2n} > t_1$. We have $B_{T_{2n}} \in \calX_{T_{2n}}$, $B_{T_{2n+1}} \in \calY_{T_{2n+1}}$ and $B_{t} \notin \calY_{t}$ for $t\in [T_{2n}, T_{2n+1})$. 
Note that $\calY$ can move during the interval $ [T_{2n}, T_{2n+1})$ only if it is pushed by $\calX$ because it is not pushed by $B$. Simple geometry shows that $\calY$ is being pushed by $\calX$ at a time $t$ only if $B$ is pushing $\calX$ at the time $t$ and $\dist(B_t, Y_t) \geq \sqrt{5}$. 
Suppose that $\dist(B_t, Y_t) \geq \sqrt{5}$ for some $t\in [T_{2n}, T_{2n+1})$ and
let $t_2 = \sup\{t\in  [T_{2n}, T_{2n+1}): \dist(B_t, Y_t) \geq \sqrt{5}\}$.
Then $Y_{t_2} = Y_{T_{2n+1}}$ but this is a contradiction since
\begin{align*}
\dist(Y_{t_2} , Y_{T_{2n+1}})
&\geq \dist(Y_{t_2}, B_{t_2}) - \dist(B_{t_2}, B_{T_{2n+1}})
- \dist(Y_{T_{2n+1}}, B_{T_{2n+1}})\\
&\geq \sqrt{5} - \eps - 1 >0.
\end{align*}
We conclude that $\dist(B_t, Y_t) < \sqrt{5}$ for all $t\in [T_{2n}, T_{2n+1})$ and, consequently, $Y_{t} = Y_{T_{2n+1}}$ for all $t\in [T_{2n}, T_{2n+1})$. This implies that for all $t\in [T_{2n}, T_{2n+1})$,
\begin{align}\label{au13.1}
\dist(Y_t, B_t) \leq \dist(Y_t, Y_{T_{2n+1}}) + \dist ( B_{T_{2n+1}}, Y_{T_{2n+1}})
+ \dist(B_{t}, B_{T_{2n+1}})
\leq  1+ \eps.
\end{align}
Let $t_3(s) = \sup\{t\in [T_{2n}, s]:
B_t \in \calX_t\}$ for $s\in [T_{2n}, T_{2n+1})$. Then, for $s\in [T_{2n}, T_{2n+1})$,
\begin{align}\label{au13.2}
\dist(B_s, X_s) \leq \dist(B_{t_3(s)}, X_{t_3(s)})
+ \dist(B_{t_3(s)}, B_s) \leq 1 + \eps.
\end{align}
By analogy, \eqref{au13.1} and \eqref{au13.2} hold on every interval of the form $ [T_{2n+1}, T_{2n+2})$ for $2n\geq n_1$. Since $\eps$ can be taken arbitrarily small,
\begin{align}\label{au13.4}
\lim_{t\uparrow T_\infty} \dist(B_t, X_t) = \lim_{t\uparrow T_\infty} \dist(B_t, Y_t) = 1.
\end{align}
This implies that $\lim_{t\uparrow T_\infty} \dist(X_t, Y_t) = 2$. 
It is easy to see that this claim and \eqref{au13.4} are also true in the case represented by \eqref{au13.5} because in that case we have $B_{T_\infty} \in  \calX_{T_\infty} \cap  \calY_{T_\infty}$.

Note that the definition \eqref{j12.3} of the local time $L^X_t$ applies in the new context of one Brownian particle and two balls, for $t< T_\infty$.
Let $L^Y_t$ be the ``local time'' of $B_t$ on $ \calY_t$, defined in a way analogous to $L^X_t$.
Suppose that for some $s_1 < T_\infty$, 
\begin{align}\label{j12.4}
L^X_{T_\infty} - L^X_{s_1} = L^Y_{T_\infty} - L^Y_{s_1} =
0.
\end{align}
Then it follows from \eqref{j12.3} and an analogous formula for $Y_t$ that $\calX_t = \calX_{T_\infty}$ and $\calY_t = \calY_{T_\infty}$ for all $t\in[s_1, T_\infty]$. 
Since $\{B_t\} \cap (\calB(X_t,1) \cup \calB(Y_t,1))=\emptyset$ for $t\in[s_1, T_\infty]$ and the circles $ \calX_{T_\infty}$ and $\calY_{T_\infty}$ are tangent at
$B_{T_\infty}$, it follows that for some $s_2 \in(s_1, T_\infty)$, the piece of Brownian path $\{B_t, t\in [s_2, T_\infty]\}$ stays inside a cone with vertex $B_{T_\infty}$ and opening $\pi/8$. This event has probability 0 according to a theorem in
\cite{KBcones}.

Next consider the case when \eqref{j12.4} is false. Let $s_3 < T_\infty$ be so large that $\dist(B_t, X_t) \leq 1.01$ and $\dist(B_t, Y_t) \leq 1.01$
for $t\in[s_3, T_\infty]$. Simple geometry shows that if $\dist(B_t, X_t) \leq 1.01$, $\dist(B_t, Y_t) \leq 1.01$, and $B$ is pushing $\calX$  or $\calY$ at time $t$ then the distance between $X$ and $Y$ is increasing.
We have assumed that  \eqref{j12.4} is false so the amount of push on the interval $[s_3, T_\infty]$ is strictly positive and we conclude that  $\liminf_{t\uparrow T_\infty}\dist(X_t, Y_t) >2$ on the event $\{T_\infty<\infty\}$, contradicting our earlier claim.
\end{proof}

Lemma \ref{au13.3} completes the construction of our model on $\R^d$.

\subsection{Processes on a torus}\label{torus}

We will denote the flat $d$-dimensional torus $\R^d/(r\Z^d)$ with edge length $r$ by $\calT ^d_r $. Technical difficulties with the definition of the joint distribution of $B,X$ and $Y$ on $\R^d$ are local in nature so we will not go into details of the analogous construction on $\calT^d_r$. We limit ourselves to the remark that if $r > 4$ then the evolution of the process $\{(B_t, X_t, Y_t), t\geq 0\}$ on $\calT^d_r$ is the natural analogue of the evolution of this process on $\R^d$.

We will say that $x= y $ (mod $r$) for $x,y \in \R^d$ and $r\in \R$ if $x= y + r z$ for some $z\in \Z^d$. We can identify $\calT_r^d$ with $[0,r)^d $ in the obvious way. If Brownian motion $B_t$ is defined on $\calT^d_r$ then we define its ``unfolded'' versions $\fB_t$ by requiring that it is a  continuous process on $\R^d$, $\fB_0 = B_0$ and $\fB_t = B_t $ (mod $r$) for all $t\geq 0$. Processes $\fX_t$ and $\fY_t$ are defined in an analogous way.

\bigskip

Recall the assumptions that $|B_0 - X_0| \geq 1$, $|B_0 - Y_0| \geq 1$, $|X_0 - Y_0| \geq 2$, and $B_0 \notin  \calX_0 \cap  \calY_0$ from Section \ref{prelim1}. We will always make these assumptions for processes $(B_t, X_t, Y_t)$ on $\R^d$ and $ \calT^d_r$.

\section{Main results}\label{main}

Our first theorem is concerned with the motion of a single ball pushed by Brownian motion in two dimensions. Since the two-dimensional Brownian motion is recurrent, we do not need to place the processes in a torus to obtain a meaningful result. On the technical and conceptual side, the proof of the first theorem is much simpler than those of the other main results so it is natural to state and prove this result first.

Let $L^X$ be defined as in \eqref{j12.3} and let
$\sigma^X_t = \inf\{s \geq 0: L^X_s \geq t\}$. 

\begin{theorem}\label{m22.2}
Suppose that $d=2$, $X_0=0$, $|B_0 - X_0| \geq 1$ and $\{(B_t, X_t), t\geq 0\}$ is defined on $(\R^2)^2$
as in \eqref{m9.2}-\eqref{j12.3}. The processes $\{ n^{-1/2} X_{ \sigma^X_{nt}} , t\geq 0 \}$ converge weakly to Brownian motion as $n\to \infty$.
\end{theorem}

Recall the definition of the vector of three process $(B_t, X_t, Y_t)$ on
a torus from Section \ref{torus}. We will sometimes emphasize the dependence on $r$ in the notation by writing $(B^r_t, X^r_t, Y^r_t)$ and similarly $(\fB^r_t, \fX^r_t, \fY^r_t)$.
Let $L^X$ be defined as in \eqref{j12.3}, let $L^Y$ be defined in an analogous way, and let $L_t = L^X_t + L^Y_t$ for $t\geq 0$. Let $\sigma_t$ be the inverse local time, i.e., $\sigma_t = \inf\{s \geq 0: L_s \geq t\}$.
Let $C_d = \sqrt {  (d-1 ) d }$ and note that $C_d >0$ for $d\geq 2$. We have $C_d =\sqrt{2}$ for $d=2$ so the normalization in the following theorem seems to contradict the normalization in Theorem \ref{m22.2}. There is no contradiction because the two theorems use different local time clocks; the local time clock in the next theorem is twice as fast as that in Theorem \ref{m22.2}, on average.

\begin{theorem}\label{j26.2}
Suppose that $d\geq 2$, and for each $r>4$ we have $|B^r_0 - X^r_0| \geq 1$, $|B^r_0 - Y^r_0| \geq 1$, $|X^r_0 - Y^r_0| \geq 2$ and $B^r_0 \notin  \calX^r_0 \cap  \calY^r_0$. Let $\{(B^r_t, X^r_t, Y^r_t), t\geq 0\}$ be defined on $\calT^d_r$. The processes $\{C_d\, n^{-1/2} (\fX^r_{ \sigma_{nt}} -\fX^r_0, \fY^r_{\sigma_{nt}} - \fY^r_0), t\geq 0 \}$ converge weakly to standard $2d$-dimensional Brownian motion when $n\to \infty$ and $r\to\infty$.
\end{theorem}

In two dimensions, the last theorem could be stated for 
$\{(B_t, X_t, Y_t), t\geq 0\}$ defined on $(\R^2)^3$. The proof of Theorem \ref{j26.2} would apply almost verbatim in 
that setting so this version of the theorem is omitted.

Theorems \ref{m22.2}-\ref{j26.2} are concerned with processes run with the local time clock as it is more meaningful than the standard clock in this context. The next theorem is stated for processes run with the usual clock but the proof shows that the result is equally true for processes run with the local time clock.

Let $\nu^d_r$ be the uniform probability measure on $\calT_r^d$. 

\begin{theorem}\label{j26.3}
Suppose that $d\geq2$, $r>4$, and $\{(B^r_t, X^r_t, Y^r_t), t\geq 0\}$ is defined on $\calT^d_r$.

(i) The process $(B^r_t, X^r_t, Y^r_t)$ has a unique stationary distribution $\calL^d_r$.

(ii) The distributions of $(X^r_0,Y^r_0)/r$ under $\calL^d_r$ converge to $\nu^d_1 \times \nu^d_1$ as $r\to \infty$.

\end{theorem}

It follows immediately from translation invariance of the process $(B^r_t, X^r_t, Y^r_t)$ that the distribution of $X^r_0$ under $\calL^d_r$ is $\nu^d_r$, and the same remark applies to $Y^r_0$. Hence, the essence of Theorem \ref{j26.3} (ii) is that the two components of $(X^r_0,Y^r_0)$ are asymptotically independent.

We will explain, in an informal way, why Theorem \ref{j26.3} does not immediately follow from Theorem \ref{j26.2}. Suppose that processes $X$ and $Y$ satisfy the conclusion of Theorem \ref{j26.2}. It is conceivable that the process $|X^r_{ \sigma_t} - Y^r_{\sigma_t}|$ on $\calT^d_r$ has a positive drift of order $1/r$ because this drift would disappear in the limit theorem for $\{n^{-1/2} (X^r_{n \sigma_t} - Y^r_{n\sigma_t}), t\in [0, t_1] \}$ for any fixed $t_1 < \infty$. The process $X^r_{ \sigma_t} - Y^r_{\sigma_t}$ needs about $r^2$ units of time to reach the stationary distribution. On this time scale, the drift of size $1/r$ would move $X_{ \sigma_t} $ and $ Y_{\sigma_t}$ about $r^2/r = r$ units apart, relative to the analogous situation without a drift. Since the effect of the drift on this time scale is comparable with the diameter of the torus, it is conceivable that under the stationary distribution $X^r$ and $Y^r$ would be typically farther apart than two random vectors with the joint distribution  $\nu^d_r \times \nu^d_r$.

\section{Excursion theory}\label{sec:exc}

This section contains a brief review of excursion theory needed
in this paper. See, e.g., \cite{M} for the foundations of the
theory in the abstract setting and \cite{B2} for the special
case of excursions of Brownian motion. Although \cite{B2} does
not discuss reflected Brownian motion, all results we need from
that book readily apply in the present context. 

Let $\P^{x}$ denote the distribution of Brownian motion
starting from $x$ and let $\E^{x}$ be the
corresponding expectation. For a domain (open connected set) $D\subset \R^d$, let $\P^x_D$ denote the distribution
of Brownian motion starting from $x\in D$ and killed upon
exiting $D$.

An ``exit system'' for excursions of reflected Brownian
motion $Z$ from $\prt D$ is a pair $(L^*_t, H^x)$ consisting of
a positive continuous additive functional $L^*_t$ of $Z$
and a family
of ``excursion laws'' $\{H^x\}_{x\in\prt D}$.
 Let $\bDelta$ denote the ``cemetery''
point outside $\ol D$ and let $\calC$ be the space of all
functions $f:[0,\infty) \to \ol D\cup\{\bDelta\}$ which are
continuous and take values in $\ol D$ on some interval
$[0,\zeta)$, and are equal to $\bDelta$ on $[\zeta,\infty)$.
For $x\in \prt D$, the excursion law $H^x$ is a $\sigma$-finite
(positive) measure on $\calC$, such that the canonical process is
strong Markov on $(t_0,\infty)$, for every $t_0>0$, with the
transition probabilities
$\P^\centerdot_D$.
 Moreover, $H^x$ gives zero
mass to paths which do not start from $x$. We will be concerned
only with the ``standard'' excursion laws; see Definition 3.2
of \cite{B2}. For every $x\in \prt D$ there exists a unique
standard excursion law $H^x$ in $D$, up to a multiplicative
constant.

Excursions of $Z$ from $\prt D$ will be denoted $e$ or $e_s$,
i.e., if $s< u$, $Z_s,Z_u\in\prt D$, and $Z_t \notin \prt D$
for $t\in(s,u)$ then $e_s = \{e_s(t) = Z_{t+s} ,\,
t\in[0,u-s)\}$ and $\zeta(e_s) = u -s$. By convention, $e_s(t)
= \bDelta$ for $t\geq \zeta(e_s)$, so $e_t \equiv \bDelta$ if
$\inf\{s> t: Z_s \in \prt D\} = t$.

Let $\sigma_t = \inf\{s\geq 0: L^*_s \geq t\}$ and
$\calE_u = \{e_s: s < \sigma_u\}$.
Let $I$ be
the set of left endpoints of all connected components of $(0,
\infty)\setminus \{t\geq 0: Z_t\in \partial D\}$. The following
is a special case of the exit system formula of \cite{M}. For
every $x\in \ol D$, every bounded predictable process $V_t$ and
every positive universally measurable function $f:\, \calC\to[0,\infty)$
that vanishes on
excursions $e_t$ identically equal to $\bDelta$, we have
\begin{align}\label{exitsyst}
\E^x \left[ \sum_{t\in I} V_t \cdot f ( e_t) \right]
= \E^x \int_0^\infty V_{\sigma_s} H^{Z(\sigma_s)}(f) ds
= \E^x \int_0^\infty V_t H^{Z_t}(f) dL^*_t.
\end{align}
Here and
elsewhere $H^x(f) = \int_\calC f dH^x$. Intuitively speaking,
\eqref{exitsyst} says that the right continuous version
$\calE_{t+}$ of the process of excursions is a Poisson
point process on the local time scale with variable intensity $H^\centerdot$.

The normalization of the exit system is somewhat arbitrary. For
example, if $(L^*_t, H^x)$ is an exit system and
$c\in(0,\infty)$ is a constant then $(cL^*_t, (1/c)H^x)$ is
also an exit system. One can even make $c$ dependent on
$x\in\prt D$. Theorem 7.2 of \cite{B2} shows how to choose a
``canonical'' exit system; that theorem is stated for the usual
planar Brownian motion but it is easy to check that both the
statement and the proof apply to reflected Brownian motion.
According to that result, we can take $L^*_t$ to be the
continuous additive functional whose Revuz measure is a
constant multiple of the surface area measure $dx$
on $\prt D$ and
$H^x$'s to be standard excursion laws normalized so that
\begin{align}\label{eq:M5.2}
H^x (A) = \lim_{\delta\downarrow 0}
\frac 1  \delta \P_D^{x + \delta\n(D,x)} (A),
\end{align}
for any event $A$ in the $\sigma$-field generated by the process
on an interval $[t_0,\infty)$, for any $t_0>0$. 

Recall the local time $L^X$ from the Skorokhod representation of reflected Brownian motion given in \eqref{m9.2}. In the present context $L^X$ will be called $L^Z$.
The Revuz
measure of $L^Z$ is the measure $dx/(2|D|)$ on $\prt D$, i.e.,
if the initial distribution of $Z$ is the uniform probability
measure $\mu$ on $D$,
then 
\begin{align}\label{jl31.1}
\E^\mu \int_0^1 \bone_A (Z_s) dL^Z_s
= \int_A dx/(2|D|)
\end{align}
for any Borel set $A\subset \prt D$. It has
been shown in \cite{BCJ} that $L^*_t=L^Z_t$, i.e., $(L^Z_t, H^x)$ is an exit system if excursion laws $H^x$ are defined as in \eqref{eq:M5.2}.

\subsection{Excursions crossing spherical shells}
We will calculate the ``probability'' under excursion law that an excursion starting at the inner boundary of a spherical shell hits the outer boundary.
Let $D=\calB(0, b) \setminus \ol{\calB(0,1)}$ for some $b>1$.

Let $A$ be the event that the process hits $\calS(0,b)$ before hitting $\calS(0,1)$. The function $x\to \P^x(A)$ is harmonic in $D$ with boundary values 1 on  $\calS(0,b)$ and 0 on $\calS(0,1)$.

In the two-dimensional case we have 
$\P^x(A) = \log|x| /\log b$ for $x\in D$ so for $x\in \calS(0,1)$,
\begin{align}\label{m9.3}
H^x (A) = \lim_{\delta\downarrow 0}
\frac 1  \delta \P_D^{x + \delta\n(D,x)} (A)
= 1/\log b.
\end{align}
For $d\geq 3$, $\P^x(A) = (1-|x|^{2-d} )/(1-|b|^{2-d} )$ for $x\in D$ so for $x\in \calS(0,1)$,
\begin{align}\label{a3.2}
H^x (A) = \lim_{\delta\downarrow 0}
\frac 1  \delta \P_D^{x + \delta\n(D,x)} (A)
= \frac{d-2} {1-|b|^{2-d} }.
\end{align}
The last formula holds with $b=\infty$. We have in that case
\begin{align}\label{m26.4}
H^x (A) = \lim_{\delta\downarrow 0}
\frac 1  \delta \P_D^{x + \delta\n(D,x)} (A)
= d-2.
\end{align}

\subsection{Expected lifetimes of excursions}\label{explif}

We will derive estimates for expected excursion lifetimes in the exterior of two balls in a torus.

Let $s_d = 2 \pi^{d/2} / \Gamma(d/2)$ denote the $(d-1)$-dimensional area of $\calS(0,1) \subset \R^d$.

\begin{lemma}\label{au4.1}
Suppose that $d\geq 2$, $r> 4$, $x_1, x_2 \in \calT^d_r$, and $D = \calT^d_r \setminus (\ol{\calB(x_1,1)} \cup \ol{\calB(x_2, 1)})$. 
Let $(L_t, H^x)$ denote the exit system for reflecting Brownian motion in $D$, normalized as in \eqref{eq:M5.2}.

(i)
There exist $r_1 $ and $c_1$ such that if $|x_1 - x_2| \geq 2$, $r \geq r_1$ and $x\in \calS(x_1,1) \cup \calS(x_2, 1)$, then
\begin{align}\label{au2.3}
 H^x(\zeta) \leq c_1 r^d .
\end{align}

(ii)
For all $\eps >0$ there exist $r_1,r_2 < \infty $ such that  if $|x_1 - x_2| \geq r_1$, $r \geq r_2$ and $x\in \calS(x_1,1) \cup \calS(x_2, 1)$, then
\begin{align*}
(1-\eps) r^d /s_d < H^x(\zeta) < (1+\eps) r^d /s_d.
\end{align*}

\end{lemma}

\begin{proof}
{\it Step 1}.
In this proof, $B$ will denote reflected Brownian motion in $ D$.
Let $\wt D:=\calT^d_r \setminus \ol{\calB(x_1,1)}$ and let $\wt H^x$ denote an excursion law in $\wt D$ normalized as in \eqref{eq:M5.2}. We will consider various ``large'' spheres, balls, etc. We will always tacitly assume that their diameters are smaller than $r$, the edge of the torus.

Every positive harmonic function in $\wt D$ is integrable by the results of \cite{Arm}. Since the density of the expected occupation measure for an excursion law for reflected Brownian motion in $\wt D$ is (a constant multiple of) the Poisson (Martin) kernel, it follows that for some $c_2 = c_2(r)<\infty$, $\wt H^x(\zeta)< c_2$ for every $x\in \calS(x_1,1)$. Since $D\subset \wt D$, $H^x(\zeta) \leq \wt H^x(\zeta)< c_2 < \infty$ for $x\in \calS(x_1,1)$.

For an excursion $e_t$,
let $S_a = \inf\{s\geq 0: e_t(s) \in \calS(x_1, a)\}$. The argument given in the last paragraph implies that for every $a >1$ there exists $\alpha(a) <\infty$ such that
for all $x\in \calS(x_1,1)$,
\begin{align}\label{s27.1}
\wt H^x(\zeta \land S_a) = \alpha(a).
\end{align}

Standard arguments show that \eqref{jl31.1} implies that, a.s.,
\begin{align}\label{au14.2}
\lim_{t\to \infty} L_t/t = 2 s_d/ (2|D|) = s_d/|D|,
\end{align}
and
\begin{align}\label{au2.7}
|D|/s_d = \lim_{t\to \infty} t/L_t
= \lim_{t\to \infty} \sigma_t/t
= \lim_{t\to \infty} \E \sigma_t/t
= \lim_{t\to \infty} \frac 1 t \int_0^t H^{B(\sigma_s)}(\zeta) ds.
\end{align}
The middle expression in \eqref{au14.2} contains 2 in front of $s_d/ (2|D|)$ because the boundary of $D$ consists of two spheres.

\medskip
{\it Step 2}.
Let $D_k = \calB(x_1, 2^k) \setminus \calB(x_1,1) $ and let $f_k(x,\,\cdot\,)$ be the density with respect to the surface area measure $\mu_k$ on $\prt D_k$ of the hitting distribution on $\prt D_k$ of Brownian motion starting from $x \in D_k$. 
We apply the Harnack inequality in the domain \break $\calB(x_1, 2^{k-1} \cdot 3/2)
\setminus \calB(x_1, 2^{k-1}\cdot 3/4)$ to see that there exists $c_3>0$, independent of $k$, such that for all $z_1,z_2\in \calS(x_1, 2^{k-1})$ and $y\in \prt D_k$, 
\begin{align}\label{au2.1}
 f_k(z_1, y)/ f_k(z_2, y) \geq   c_3.
\end{align}
By the strong Markov property applied at the hitting time of
$\calS(x_1, 2^{k-1})$, for $z\in D_k \cap \calB(x_1, 2^{k-1})$ and $y\in \calS(x_1, 2^k) $, 
\begin{align}\label{au2.2}
f_k(z, y) = \int_{\calS(x_1, 2^{k-1})} f_k(x,y) f_{k-1}(z,x) \mu_{k-1}(dx).
\end{align}
We now apply \cite[Lem.~6.1]{BTW} (see \cite[Lem.~1]{BKhosh} for a better presentation of the same estimate) to see that \eqref{au2.1}-\eqref{au2.2} imply that 
there exist constants $C_j$, $1\leq j\leq k-1$, such that for every $1\leq j\leq k-1$ and all 
$z_1,z_2\in \calS(x_1, 2^{k-j})$ and $y\in \calS(x_1, 2^k)$,
\begin{align*}
 f_k(z_1, y)/ f_k(z_2, y) \geq   C_j.
\end{align*}
Moreover, $C_j\in(0,1)$, $C_j$'s depend only on $c_3$, and $1-C_j \le e^{-c_4 j}$ for some $c_4>0$ and all $j$. Hence, for $z_1,z_2\in \calS(x_1, 2)$ and $y\in \calS(x_1, 2^k)$,
\begin{align}\label{au2.11}
 f_k(z_1, y)/ f_k(z_2, y) \geq   C_{k-1} \geq 1 - c_{5} e^{-c_4 k}.
\end{align}

\medskip
{\it Step 3}.
We will first prove part (ii) of the lemma.
Fix an arbitrarily small $\eps >0$. 
Assume that $|x_1 - x_2| > 2^{m+1}$
and choose $m$ so large that, in view of \eqref{au2.11} and the strong Markov property applied at the time $S_2$, 
\begin{align}\label{au2.12}
1-\eps \leq
\frac
{H^{z_1}( e(S_{ 2^m}) \in dy) }
{H^{z_2}( e(S_{ 2^m} )\in dy) }
=
\frac
{\wt H^{z_1}( e(S_{ 2^m}) \in dy) }
{\wt H^{z_2}( e(S_{ 2^m})\in dy) }
\leq 1+\eps,
\end{align}
for all $z_1,z_2\in \calS(x_1,1)$ and $y\in \calS(x_1, 2^m)$.

Let $T_A = \inf\{t\geq 0: B_t \in A\}$ for any set $A$.
We will now treat $m$ as a fixed number (we will consider $r$ a variable and we will let $r\to \infty$) so that \eqref{s27.1} yields 
$H^x(\zeta \land S_{2^m}) = \wt H^x(\zeta \land S_{2^m}) = c_6$ for some $c_6 < \infty$ and all $x\in \calS(x_1,1)$. We obtain for $z_1, z_2 \in \calS(x_1,1)$,
\begin{align*}
H^{z_1}(\zeta) 
& = H^{z_1}(\zeta \land S_{2^m})
+ H^{z_1}(\zeta - S_{2^m}; \zeta \geq S_{2^m})\\
&=  c_{6}
+ \int_{\calS(x_1,2^m)}
\E^y T_{\calS(x_1,1)\cup \calS(x_2,1)} H^{z_1} (e(S_{2^m}) \in dy; \zeta \geq S_{2^m})\\
&\leq c_{6 } + (1+\eps)
\int_{\calS(x_1,2^m)}
\E^y T_{\calS(x_1,1)\cup \calS(x_2,1)} H^{z_2} (e(S_{2^m}) \in dy; \zeta \geq S_{2^m}).
\end{align*}
Let $\beta = \int_{\calS(x_1,2^m)}
\E^y T_{\calS(x_1,1)\cup \calS(x_2,1)} H^{z_0} (e(S_{2^m}) \in dy; \zeta \geq S_{2^m})$ for some arbitrarily chosen $z_0\in \calS(x_1,1)$.
The last estimate may be now written as
\begin{align*}
H^{z}(\zeta) \leq c_{6 } + (1+\eps) \beta,
\end{align*}
for $z\in \calS(x_1,1)$. By symmetry, the estimate holds also for all $z\in\calS(x_2,1)$.
We can use the lower bound in \eqref{au2.12} to derive the analogous lower estimate, so for $z\in \prt D$,
\begin{align}\label{au2.13}
c_{6 } + (1-\eps) \beta \leq H^{z}(\zeta) \leq c_{6 } + (1+\eps) \beta.
\end{align}

Note that $|D|/r^d \to 1$ as $r\to \infty$. 
By \eqref{au2.7} and \eqref{au2.13}, for large $r$,
\begin{align}\label{au2.14}
(1-\eps)r^d/s_d \leq \lim_{t\to \infty} \frac 1 t \int_0^t H^{B(\sigma_s)}(\zeta) ds
\leq \lim_{t\to \infty} \frac 1 t \int_0^t (c_{6 } + (1+\eps) \beta) ds.
\end{align}
For similar reasons,
\begin{align}\label{au2.15}
(1+\eps)r^d/s_d \geq \lim_{t\to \infty} \frac 1 t \int_0^t H^{B(\sigma_s)}(\zeta) ds
\geq \lim_{t\to \infty} \frac 1 t \int_0^t (c_{6 } + (1-\eps) \beta) ds.
\end{align}
The estimates \eqref{au2.14}-\eqref{au2.15} show that, 
\begin{align*}
\frac {1-\eps}{1+\eps} r^d/s_d - \frac{c_{6}}{1+\eps}\leq \beta
\leq \frac {1+\eps}{1-\eps}r^d/s_d - \frac{c_{6}}{1-\eps}.
\end{align*}
Since $\eps>0$ is arbitrarily small and $r$ can be made large, this and \eqref{au2.13} imply part (ii) of the lemma.

Next we prove part (i) of the lemma. Note that the argument given in this step
remains valid if we replace all occurrences of excursion laws $H$ with $\wt H$
and drop the assumption on the distance between $x_1$ and $x_2$.
The only difference is that we would need \eqref{au2.7} with an extra constant 2 on the left hand side because we would use the exit system in $\wt D$ rather than $D$.
Hence, we have $\wt H^x(\zeta) \leq c_7 r^d$ for large $r$, that is, an inequality analogous to the upper estimate in part (ii). Part (i) follows because $H^x(\zeta) \leq \wt H^x(\zeta)$.
\end{proof}

\section{Motion of a single ball}\label{single}

\begin{proof}[Proof of Theorem \ref{m22.2}]

Our starting point is the following equation, analogous to \eqref{m9.2}, where $\calX_0$ is replaced with $\calS(0,1)$. We will consider a continuous process $Z_t$ (reflected Brownian motion) taking values in $\calB(0,1)^c$ and local time $L^Z_t$ such that $L^Z_0 = 0$, $\int_0^\infty \bone_{\{Z_t \notin \calS(0,1)\}} dL^Z_t = 0$, and
\begin{align}\label{m22.1}
Z_t = B_t + \int_0^t \n( \calS(0,1), Z_s) dL^Z_s, \qquad t\geq 0.
\end{align}
Let $\sigma^Z_t$ be the inverse local time, i.e., $\sigma^Z_t = \inf\{s\geq0: L^Z_s \geq t\}$, and
\begin{align}\label{au5.3}
\left(\bL^{1,Z}_t,  \bL^{2,Z}_t\right) =
\bL^Z_t &= \int_0^t \n( \calS(0,1), Z_s) dL^Z_s.
\end{align}
It follows from the recurrence of two-dimensional Brownian motion that $\lim_{t\to \infty} L^Z_t = \infty$. 

Comparing the above setup with the statement of the theorem and 
\eqref{m9.2}-\eqref{j12.3}, we conclude that it will suffice to prove that processes
$\{ n^{-1/2} \bL^Z_{n \sigma^Z_t} , t\geq 0 \}$ converge weakly to Brownian motion as $n\to \infty$.

We will first compute the distribution of $Z_{\sigma^Z_t}$ assuming that $Z_0=(0,1)$. 
Let $K_a=\{(x^1,x^2)\in \R^2: x^1 < a\}$.
Let $W_t= (W^1_t, W^2_t) $ be the reflected Brownian motion in the half-plane $K_0$ starting from $W_0=0$ and let $ L^W_t$ be its local time on $\prt K_0$. Let $\sigma^W_t = \inf\{s\geq0:  L^W_s \ge t\}$. Suppose that $B_0=0$ and let $(a,M_a) $ be the (random) location of $B$ at the hitting time of the line $\prt K_a$. It is well known that the processes $\{W^2_{ \sigma^W_t}, t\geq 0\}$  and $\{M_t, t\geq 0\}$ have the same distribution. Hence, for a fixed $t$, the distribution of $W_{ \sigma^W_t}$ is the same as the harmonic measure on $\prt K_t$ with the base point at $(0,0)$. The complex analytic function $z\to \exp (-z)$ maps $W$ onto a time change of $Z$ and the local time is conformally invariant, so $Z_{\sigma^Z_t}$ is distributed
as the image of the distribution of $W_{ \sigma^W_t}$ under the map $z\to \exp (-z)$. By the earlier remarks and conformal invariance of harmonic measure under the map $z\to \exp (z-t)$, the distribution of $Z_{\sigma^Z_t}$ is the same as the  harmonic measure in $\calB(0,1)$ relative to the base point $\exp(-t)$.
The density of this harmonic measure with respect to the uniform probability measure on $\calS(0,1)$ at a point $z\in\calS(0,1)$ is 
\begin{align}\label{m23.1}
\frac{1- e^{-2t}}{|z-(e^{-t},0)|^2},
\end{align}
 by the Poisson formula (see \cite[Ch.~4, Sect.~6.3]{Ahlf}).

Let $\theta = \arg(z) \in [0, 2\pi)$ for $z\in \calS(0,1)$ and let $\mu_t(dz)$ be the distribution of $Z_{\sigma^Z_t}$.
We will need a formula for $\int_{\calS(0,1)} \cos (\theta) \mu_t(dz)$. Note that for any fixed $\vphi$, the function $z \to \cos(\theta-\vphi)$ is equal to the harmonic function $z=(z^1,z^2) \to z^1\cos(\vphi) + z^2 \sin(\vphi)$ on $\calS(0,1)$. Since $\mu_t(dz)$ is the harmonic measure with the base point $\exp(-t)$, it follows that 
\begin{align}\label{m23.2}
\int_{\calS(0,1)} \cos (\theta-\vphi) \mu_t(dz) = \exp(-t)\cos(\vphi).
\end{align}

Recall the normal vector $\n$ and write
$\n = (\n_1, \n_2)$. 
The strong Markov property applied at $\sigma^Z_t$, invariance of the transition probabilities of $Z$ under rotations about 0, and \eqref{m23.2} imply for $s>t$,
with the convention $\theta = \arg(z)$,
\begin{align}\label{s27.2}
\E\left(\n_1( \calS(0,1), Z_{\sigma^Z_s}) \mid  Z_{\sigma^Z_t} = e^{i\vphi}\right)
= \int_{\calS(0,1)} \cos (\theta-\vphi) \mu_{s-t}(dz) = \exp(t-s)\cos(\vphi).
\end{align}
Let $\calU_x$ denote the uniform probability distribution on $\calS(x,1)$. 
Assume that $B_0$ has the distribution $\calU_0$. Then $Z_{\sigma^Z_t}$ has the same distribution.
Hence, by  \eqref{s27.2},
\begin{align*}
\E&\left(\n_1( \calS(0,1), Z_{\sigma^Z_s})\ \n_1( \calS(0,1), Z_{\sigma^Z_t})\right)
= \int_{\calS(0,1)} \exp(t-s)\cos^2(\theta) \calU_0(dz)\\
& = \int_0^{2 \pi}\frac 1 {2 \pi}\exp(t-s) \cos ^2(\theta) d\theta 
= (1/2)\exp(t-s).
\end{align*}
It follows that
\begin{align}\label{au14.8}
\E\left(\bL^{1,Z}_{\sigma^Z_u}\right)^2 
&= 
\E\left( \int_{\sigma^Z_0}^{\sigma^Z_u} \n_1( \calS(0,1), Z_s) dL^Z_s  \right)^2
= \E\left( \int_0^u \n_1( \calS(0,1), Z_{\sigma^Z_s}) ds  \right)^2\\
&= 2 \int_0^u \int_t^u \E\left(\n_1( \calS(0,1), Z_{\sigma^Z_s})\ \n_1( \calS(0,1), Z_{\sigma^Z_t})\right) ds dt\nonumber\\
&= 2 \int_0^u \int_t^u (1/2) \exp(t-s) \, ds dt = 
 u+ e^{-u} -1 .\nonumber
\end{align}

For $j=1,2$, and $n=1,2, \dots$, let
\begin{align*}
\bL^{j,Z}_{(n)} = \int_{\sigma^Z_n}^{\sigma^Z_{n+1}} \n_j( \calS(0,1), Z_s) dL^Z_s.
\end{align*}
Suppose that $B_0$ has the uniform distribution on $\calS(0,1)$. Then, for every $t$, $Z_{\sigma^Z_t}$ also has the uniform distribution on $\calS(0,1)$. By the strong Markov property, the sequence $\{\bL^{1,Z}_{(n)}, n\geq 0\}$ is strictly stationary. It follows from \eqref{m23.1} and the strong Markov property that  the sequence $\{\bL^{1,Z}_{(n)}, n\geq 0\}$ is $\vphi$-mixing in the sense of \cite[Sect.~20]{Bill} with $\vphi_n \leq c_1 e^{-n}$ for some $c_1<\infty$.
Hence, by \cite[Thm. 20.1]{Bill},  
$\{ n^{-1/2} \bL^{1,Z}_{ \sigma^Z_{\lfloor n t\rfloor}} , t\geq 0 \}$ 
converge in distribution to Brownian motion with some diffusion coefficient, as $n\to \infty$. Note that $|\bL^{1,Z}_{(n)}| \leq 1$, a.s. This implies that $\{ n^{-1/2} \bL^{1,Z}_{  \sigma^Z_{nt}} , t\geq 0 \}$ 
converge to Brownian motion in distribution, as $n\to \infty$. The same applies to 
$\{ n^{-1/2} \bL^{2,Z}_{  \sigma^Z_{nt}} , t\geq 0 \}$ and to every process of the form $\{ n^{-1/2} \bL^{Z}_{  \sigma^Z_{nt}} \cdot v , t\geq 0 \}$, for every vector $v$ of unit length. The last observation can be rephrased by saying that every linear combination of the processes $\bL^{1,Z}_{  \sigma^Z_t}$ and $\bL^{2,Z}_{  \sigma^Z_t}$ satisfies the same type of invariance principle, with possibly different normalizing constant. Applying the strong Markov property at 
$\sigma^Z_{ns}$, we can prove that for any unit vector $v$,
\begin{align*}
\left\{ \{ n^{-1/2}(\bL^{Z}_{  \sigma^Z_{n(t+s)}}- \bL^{Z}_{  \sigma^Z_{ns}}) \cdot v,   t\geq 0 \},
\{ n^{-1/2} \bL^{1,Z}_{  \sigma^Z_{nt}} , t\in[0,s] \},
\{ n^{-1/2} \bL^{2,Z}_{  \sigma^Z_{nt}} , t\in[0,s]\}
\right\}
\end{align*}
converge jointly and the first component of the limit is Brownian motion independent of the other two components. Hence, 
\begin{align*}
\left\{ n^{-1/2} (\bL^{1,Z}_{  \sigma^Z_{nt}},\bL^{2,Z}_{  \sigma^Z_{nt}}) , t\geq 0 \right\}
\end{align*}
converges to a process whose each component is Brownian motion and whose increments are Gaussian and independent. We conclude that the limit process is Gaussian and it is a constant multiple of Brownian motion, by rotation invariance.
It remains to identify the diffusion coefficient. The formula \eqref{au14.8}  implies  that $\lim_{u\to \infty} \Var \bL^{1,Z}_{\sigma^Z_u} /u=1 $.
In general, the existence of a weak limit for a sequence of random variables does not imply that the variance of the limit is the limit of variances but this is the case in the setting of  \cite[Sect.~20]{Bill} so we conclude that the diffusion coefficient is 1. 
\end{proof}

\section{Estimates for the local time}\label{local}

Recall the process $Z$ defined in \eqref{m9.2}. Let $\calX_0 = \calS(0,1)$ in that equation so that the process $Z$ takes values in $\R^d \setminus \calB(0,1)$.
We will identify $L^Z$ with $L^X$ that appeared in \eqref{m9.2}.
Let
\begin{align}\label{au14.7}
\left( \bL^{1,Z}_t, \dots,  \bL^{d,Z}_t\right) &=
 \bL^Z_t = \int_0^{t} \n( \calS(0,1), Z_s) dL^Z_s.
\end{align}
If $d\geq 3$ then \eqref{m26.4} implies that $L^Z_\infty< \infty$, a.s., so $\bL^{Z}_\infty$ is a well defined vector.

\begin{lemma}\label{m22.3}
Suppose that $d\geq 3$ and $\{(B_t, Z_t), t\geq 0\}$ is defined on $(\R^d)^2$. 
Assume that $B_0$ is distributed uniformly on $\calS(0,1)$. Then
\begin{align*}
\E (\bL^{1,Z}_\infty)^2 = \frac{2 }{(d-2 ) (d-1 ) d}.
\end{align*}
\end{lemma}
\begin{proof}
Let 
\begin{align*}
M_t &= \inf\{|B_s|: s\in [0,t]\},\\
\wt U_t &= B_t/M_t,\\
C_t &= \int_0^t M_s^{-2} ds,\\
\gamma_t &= \inf\{s \geq 0: C_s \geq t\},\\
U_t &= \wt U_{\gamma_t}.
\end{align*}
Processes $\{U_t, t\geq 0\}$ and $\{Z_t, t\geq 0\}$ have the same distribution. This claim is a slight modification of \cite[Thm.~2.3]{Pas}, where a ``scaling coupling'' was constructed. 
The above construction is related to ``perturbed Bessel processes,'' see, e.g., \cite{DWY}.

Let $L^U_t$ be the local time of $U$ on $\calS(0,1)$. Then, informally speaking, $dL^U = - dM/M$ and, therefore, $L^U_t = - \log M_t$. The last formula can be verified rigorously, for example, by using excursion theory.

Let $\be_k$ be the $k$-th vector in the usual orthonormal basis for $\R^d$ and
\begin{align*}
\sigma^U_t &= \inf\{s\geq 0: L^U_s \geq t\},\\
\sigma^M_t &= \inf\{s\geq 0: M_s = e^{-t}\}.
\end{align*}
For all $t\in (0,\infty)$, the distributions of $\sigma^U_t$ and $\sigma^M_t$ are defective because there may be no $s$ such that $L^U_s \geq t$ or $ M_s = e^{-t}$. 
The distribution of $U_{\sigma^U_t}$
is the same as the distribution of $B_{\sigma^M_t}/ |B_{\sigma^M_t}|$. 
We will give the value 0 to these and similar quantities  in our calculations
when $\sigma^U_t$ or $\sigma^M_t$ are undefined.
The distribution of $B_{\sigma^M_t}$ is the same as the hitting distribution of $\calS(e^{-t}, 0)$. Suppose that $B_0 = \be_1$. Then, by the Kelvin transformation (see \cite[Thm.~3.1, p.~102]{PS}), the distribution  of $B_{\sigma^M_t}$ is the same as the hitting distribution of $\calS(e^{-t}, 0)$ by Brownian motion starting from the point $e^{-2t} \be_1$
times the probability that Brownian motion starting from $\be_1$ will hit $\calS(0,e^{-t})$. The last probability is equal to $e^{-t(d-2)}$. Since $|B_{\sigma^M_t}| = e^{-t}$, we see that the (defective) distribution $\mu_t(dz)$ of $B_{\sigma^M_t}/ |B_{\sigma^M_t}|$ is the same as $e^{t(2-d)}$ times the hitting distribution of $\calS(0,1) $ by Brownian motion starting from the point $e^{-t}\be_1$.

For any fixed unit vector $v\in \calS(0,1)$, the function $z\to v \cdot z$ is harmonic. It follows that 
\begin{align}\label{m26.1}
\int_{\calS(0,1)} v \cdot z \,\mu_t(dz) = e^{-t} e^{t(2-d)} v \cdot \be_1
 =  e^{t(1-d)} v \cdot \be_1.
\end{align}

Let $\calU_x$ denote the uniform probability distribution on $\calS(x,1)$. 
Note that $\calU_0$ is the harmonic measure in $\calB(0,1)$ with the base point at 0.
The function 
\begin{align*}
f(x_1, \dots, x_d) =  ((d-1) x_1^2 - x_2^2 - \dots - x_d^2 +1)/d 
\end{align*}
is harmonic and its values on $\calS(0,1)$ are the same as those of the function
$(x_1, \dots, x_d) \to x_1^2 = (z \cdot \be_1)^2$ so
\begin{align}\label{m26.2}
\int_{\calS(0,1)} (z \cdot \be_1)^2\,  \calU_0(dz) 
= \int_{\calS(0,1)} f(z)\,  \calU_0(dz) 
=f(0)
= 1/d.
\end{align}

Let $\n = (\n_1, \n_2, \dots, \n_d)$ and for $j=1,\dots, d$, and $n\geq 1$,
\begin{align*}
\bL^{j,U}_{(n)} = \int_{\sigma^U_n}^{\sigma^U_{n+1}} \n_j( \calS(0,1), U_s) dL^U_s.
\end{align*}
We have by the strong Markov property applied at $\sigma^U_t$, rotation invariance of reflected Brownian motion and \eqref{m26.1}, for $s>t$ and $v \in \calS(0,1)$,
\begin{align}\label{a3.3}
\E\left(\n_1( \calS(0,1), U_{\sigma^U_s}) \mid  U_{\sigma^U_t} = v\right)
= \int_{\calS(0,1)} v\cdot z \, \mu_{s-t}(dz) = e^{(s-t)(1-d)} v \cdot \be_1.
\end{align}
Now assume that $B_0$ is uniformly distributed over $\calS(0,1)$. Then $U_{\sigma^U_t}$ has the defective distribution $e^{t(2-d)}\calU_0$.
We obtain from \eqref{m26.2} and \eqref{a3.3}, for $s>t$,
\begin{align*}
&\E\left(\n_1( \calS(0,1), U_{\sigma^U_s})\ \n_1( \calS(0,1), U_{\sigma^U_t})\right)
= \int_{\calS(0,1)} e^{(s-t)(1-d)} (z \cdot \be_1)^2 e^{t(2-d)} \calU_0(dz)\\
&\qquad = e^{(s-t)(1-d)} e^{t(2-d)}/d
=  e^{s(1-d)} e^t/d.\nonumber
\end{align*}
It follows that
\begin{align*}
\E\left(\bL^{1,U}_{\infty}\right)^2 
&= 
\E\left( \int_0^\infty \n_1( \calS(0,1), U_s) dL^U_s  \right)^2
= \E\left( \int_0^{\infty} \n_1( \calS(0,1), U_{\sigma^U_s}) ds  \right)^2\\
&= 2 \int_0^{\infty} \int_t^{\infty} \E\left(\n_1( \calS(0,1), U_{\sigma^U_s}) \ \n_1( \calS(0,1), U_{\sigma^U_t})\right) ds dt \nonumber \\
&= 2 \int_0^{\infty} \int_t^{\infty} (1/d) e^{s(1-d)} e^t\, ds dt
\nonumber \\
& = \frac{2 }{(d-2 ) (d-1 ) d}.\nonumber
\end{align*}
\end{proof}

Recall the 
notation from the beginning of this section and \eqref{au14.7}.
Consider any $b>1$ and let
\begin{align}
T_b & = \inf\{ t\geq 0: Z_t \notin  \calB(0, b)\},\nonumber \\
\lambda_1(d,b) &=
\begin{cases}
1/\log b & \text {if  } d=2,\\
\displaystyle
\frac{d-2} {1-|b|^{2-d} } & \text{if  } d\geq 3,
\end{cases} \label{au9.1}\\
\lambda_2(d,b) &=
\E ( \bL^{1,Z}_{T_b})^2 .\label{au9.2}
\end{align}
The expectation on the last line is calculated under the assumption that 
$B_0$ is distributed uniformly on $\calS(0,1)$.

\begin{lemma}\label{a3.1}
Suppose that $d\geq 2$ and $\{(B_t, Z_t), t\geq 0\}$ is defined on $(\R^d)^2$. 
Assume that $B_0$ is distributed uniformly on $\calS(0,1)$. 

(i) If $d=2$ then
\begin{align*}
 \lim_{b\to \infty} \lambda_2(d,b)/\log b = 1.
\end{align*}

(ii) If $d\geq 3$ then
\begin{align*}
\lim_{b\to \infty} \lambda_2(d,b) = \frac{2 }{(d-2 ) (d-1 ) d}.
\end{align*}

(iii) For every $\eps,\beta >0$,
\begin{align*}
\lim_{b\to \infty} b^\beta \lambda_1(d,b)
\E \left(\left( \bL^{1,Z}_{T_b}/ \sqrt{b^\beta }\right) \bone _{\left\{ \bL^{1,Z}_{T_b}/ \sqrt{b^\beta }> \eps\right\}}\right)^2
= 0.
\end{align*}
\end{lemma}
\begin{proof}

(i)
Let $d=2$, $S_1=0$,
\begin{align*}
T_k &= \inf\{t\geq S_k: |Z_t| \geq b\} , \qquad k\geq 1, \\
S_k &= \inf\{t \geq T_{k-1}: |Z_t| = 1\}, \qquad k \geq 2.
\end{align*}
Recall the notation
$\n = (\n_1, \n_2)$ and  for $j=1,2$, and $n=1,2, \dots$, let
\begin{align*}
\wt \bL^{j,Z}_{(n)} = \int_{S_n}^{T_n} \n_j( \calS(0,1), Z_s) dL^Z_s.
\end{align*}

Recall that $B_0$ has the uniform distribution on $\calS(0,1)$. Then, for every $n$, $Z_{S_n}$ also has the uniform distribution on $\calS(0,1)$. By the strong Markov property, the sequence $\{\wt \bL^{1,Z}_{(n)}, n\geq 0\}$ is strictly stationary. 
One can prove that  the sequence $\{\wt \bL^{1,Z}_{(n)}, n\geq 0\}$ is $\vphi$-mixing in the sense of \cite[Sect.~20]{Bill} with $\vphi_n \leq c_1 e^{-n}$ for some $c_1<\infty$
using the formula \eqref{m23.1} and the method employed in the proof of \eqref{au2.11}.
Let $T(t) $ be the largest $T_n \leq t$.
By \cite[Thm. 20.1]{Bill}, for some $c_1$, 
$\{c _1 n^{-1/2} \wt \bL^{1,Z}_{ \sigma^Z_{T(t)}} , t\geq 0 \}$ 
converge to Brownian motion in distribution, as $n\to \infty$. 
The distribution of $L^Z_{T_n} - L^Z_{S_n}$ is exponential with mean $\log b$, by \eqref{m9.3}. Hence
$|\wt \bL^{1,Z}_{(n)}| $ is majorized by an exponential random variable with mean $\log b$. This implies that $\{c _1 n^{-1/2} \bL^{1,Z}_{  \sigma^Z_{nt}} , t\geq 0 \}$ 
converge to Brownian motion in distribution, as $n\to \infty$. 
This we already know from the proof of Theorem \ref{m22.2}. The point of the present argument is that this time we divided the time axis into subintervals of independent lengths with exponential distributions with mean $\log b$. Since the limit process is the same in both cases, the variances must match and, therefore, $\lim_{b\to \infty} \lambda_2(d,b)/\log b = 1$ because the contributions from the cross terms will disappear in the limit, for the same reason why we have $\vphi$-mixing.

(ii) Suppose that $d\geq 3$.
It follows from \eqref{m26.4} that $L^Z_\infty$ has the exponential distribution with mean $1/(d-2)$. This and the formula \eqref{au14.7} show that the family $\{\bL^{1,Z}_{T_b}\}_{b>1}$ is uniformly integrable.
We have $\lim_{b\to \infty}\bL^{1,Z}_{T_b} = \bL^{1,Z}_\infty$, a.s., so
part (ii) of the lemma follows from Lemma \ref{m22.3}.

(iii) 
For any starting point $B_0=x \notin  \calB(0,1)$, the distribution of $Z_{\sigma^Z_1}$ is absolutely continuous with respect to the uniform probability measure on $\calS(0,1)$, according to the proofs of Theorem \ref{m22.2} and Lemma \ref{m22.3}. Let $c_1 = c_1(d)$ be the maximum of the corresponding Radon-Nikodym derivative and note that $c_1(d)<\infty$. 
Let $T'_b = T_b \circ \theta _{\sigma^Z_1}$ where $\theta$ is the usual Markov shift.
By the strong Markov property applied at ${\sigma^Z_1}$ and parts (i) and (ii) of this lemma, for all $x\in \calS(0,1)$,
\begin{align}\label{a8.1}
\E^x ( \bL^{1,Z}_{T_b})^2 
&\leq
\E^x \left( \bL^{1,Z}_{\sigma^Z_1} +(  \bL^{1,Z}_{T'_b} -  \bL^{1,Z}_{\sigma^Z_1})\right)^2\\
&\leq 2\E^x ( \bL^{1,Z}_{\sigma^Z_1})^2 
+ 2 \E^x \left(  \bL^{1,Z}_{T'_b} -  \bL^{1,Z}_{\sigma^Z_1}\right)^2
\nonumber\\
&\leq 2 + 2 c_1 \lambda_2(d,b).\nonumber
\end{align}
 Let $c_2 = \sqrt{2 + 2c_1 \lambda_2(d,b)}$. Then for large $b$, 
\begin{align*}
\P^x ( \bL^{1,Z}_{T_b} \geq 2 c_2) \leq \E^x ( \bL^{1,Z}_{T_b})^2 / (4c_2^2) \leq 1/4.
\end{align*}
An application of the strong Markov property at times $\sigma^Z_{2k  c_2}$ shows that, for $k\geq 1$,
\begin{align*}
\P^x ( \bL^{1,Z}_{T_b} \geq 2 k  c_2) \leq 1/4 ^k,
\end{align*}
and this implies that for some $c_3>0$ and all $a> 0$,
\begin{align*}
\P^x ( \bL^{1,Z}_{T_b} \geq a) \leq \exp(-c_3 a/ c_2 )
=  \exp\left(-\frac{c_3a}{\sqrt{2 + 2c_1 \lambda_2(d,b)}} \right).
\end{align*}
This implies that for every $\eps,\beta >0$,
\begin{align}\label{a5.1}
&\lim_{b\to \infty} b^\beta \lambda_1(d,b)
\E^x \left(\left( \bL^{1,Z}_{T_b}/ \sqrt{b^\beta }\right) \bone _{\left\{ \bL^{1,Z}_{T_b}/ \sqrt{b^\beta }> \eps\right\}}\right)^2\\
& = \lim_{b\to \infty} \lambda_1(d,b)
\E^x \left( \bL^{1,Z}_{T_b} \bone _{\left\{ \bL^{1,Z}_{T_b}
> \eps\sqrt{b^\beta }\right\}}\right)^2 \nonumber \\
&\leq \lim_{b\to \infty} \lambda_1(d,b)
\sum_{k: 2^{k+1} \geq  \eps\sqrt{b^\beta } }
2^{2(k+1)} \exp\left(-\frac{c_3 2^k}{\sqrt{2 + 2c_1 \lambda_2(d,b)}} \right). \nonumber
\end{align}
For a fixed $d\geq 3$, the quantities $\lambda_1(d,b)$ and $\lambda_2(d,b)$ have limits in $(0,\infty)$ as $b\to \infty$. This and the fact that the series 
\begin{align*}
\sum_{k\geq 0 }
2^{2(k+1)} \exp\left(-\frac{c_3 2^k}{\sqrt{2 + 2c_1 c_4}} \right)
\end{align*}
is summable for any $c_4\in(0,\infty)$ imply that 
the limit in \eqref{a5.1} is equal to 0. This proves part (iii) for $d\geq 3$.

\medskip

For $d=2$, $\lim_{b\to \infty} \lambda_2(d,b)/ \log b = 1$. If $2^{k+1} \geq  \eps\sqrt{b^\beta }$ and $n\geq k+1$, then the ratio of the two consecutive terms in the series on the right hand side of \eqref{a5.1}, corresponding to indices $k=n$ and $k=n+1$, is equal to
\begin{align*}
&(1/4) \exp\left(-\frac{c_3 2^n}{\sqrt{2 +  2c_1  \lambda_2(d,b)}} 
+\frac{c_3 2^{n+1}}{\sqrt{2 +  2c_1  \lambda_2(d,b)}} \right)\\
&= (1/4) \exp\left(\frac{c_3 2^n}{\sqrt{2 +  2c_1  \lambda_2(d,b)}}  \right)
\geq (1/4) \exp\left(\frac{c_3 \eps\sqrt{b^\beta }}{\sqrt{2 +  2c_1  \lambda_2(d,b)}}  \right).
\end{align*}
The last expression is greater than $1/2$ for sufficiently large $b$, so for large $b$, the series on the right hand side of \eqref{a5.1} is bounded by twice its first term, and this implies that
\begin{align*}
&\lim_{b\to \infty} b^\beta \lambda_1(d,b)
\E^x \left(\left( \bL^{1,Z}_{T_b}/ \sqrt{b^\beta }\right) \bone _{\left\{ \bL^{1,Z}_{T_b}/ \sqrt{b^\beta }> \eps\right\}}\right)^2\\
&\leq \lim_{b\to \infty} 2 \lambda_1(d,b)
 \eps^2 b^\beta  \exp\left(-\frac{c_3 \eps\sqrt{b^\beta }/2}{\sqrt{2 +  2c_1  \lambda_2(d,b)}} \right)=0. \qedhere
\end{align*}
\end{proof}

\section{Hitting distribution estimates}\label{hit}

For an open set $D$ and a point $x\in D$, let $\mu^D_x(dy)$ be the harmonic measure on $\prt D$ with the base point $x$.

\begin{lemma}\label{au14.1}
Consider $b>2$, $r> 8b$, and points $x_1, y_1 \in \calT^d_r$ with
$\dist(x_1, y_1) > 2b$.
Suppose that $z_1  \in \calS(x_1, b)\cup \calS(y_1,b)$ and let $D = \calT^d_r \setminus (\ol{\calB(x_1,1)} \cup \ol{\calB(y_1,1)})$.
There exists $c_1$ such that
\begin{align}\label{m9.1}
1- c_1 /\log b \leq 
\mu^{D}_{z_1}(\calS(x_1,1)) / \mu^{D}_{z_1}(\calS(y_1,1)) \leq 1 + c_1/\log b,
\qquad \hbox{  if  } d=2, \\
1- c_1 b^{2-d} \leq 
\mu^{D}_{z_1}(\calS(x_1,1)) / \mu^{D}_{z_1}(\calS(y_1,1)) \leq 1 + c_1 b^{2-d},
\qquad \hbox{  if  } d\geq 3. \label{au15.1}
\end{align}

\end{lemma}

\begin{proof}

Let 
\begin{align*}
A(x_1) &= \{z\in \calT^d_r: \mu^{D}_{z}(\calS(x_1,1)) \geq \mu^{D}_{z}(\calS(y_1,1))\},\\
A(y_1) &= \{z\in \calT^d_r: \mu^{D}_{z}(\calS(x_1,1)) \leq \mu^{D}_{z}(\calS(y_1,1))\}.
\end{align*}
We have $A(x_1) \cup A(y_1) = \calT^d_r$ so the volume of each of these sets is equal to or greater than $r^d/2$ (the volumes are equal by symmetry). 
We define some subsets of $\R^d$ as follows,
\begin{align*}
\bA(x) &= A(x) + r \Z^d \ 
\text{  for  } x=x_1,y_1; \qquad
\bS(x,1) =  \calS(x,1)+ r \Z^d \ 
 \text{  for  } x\in \calT^d_r.
\end{align*}

The problem is invariant under translations so we may and will assume that $x_1$ and $y_1$ are positioned in such a way in $\calT^d_r = [0,r)^d \subset \R^d$  that $\calS(x_1,b) \cup \calS(y_1,b)$ does not intersect the boundary of $[0,r)^d$.

Let $B$ denote Brownian motion in $\R^d$ starting from $z_1 \in \calS(x_1,b) \cup \calS(y_1,b)$ and for $K\subset \R^d$ let $T^B(K)$ be the first hitting time of $K$.
Let $U_1 = 0$ and $\bT'_1 = \calT_r^d$, and note that $B_{U_1} \in \bT'_1$. Let $\bT_1$ be a $d$-dimensional cube in $\R^d$ with the same center as $\bT'_1$ but with edge length equal to $3r$.
For $k\geq 2$, let $U_k = \inf\{t\geq U_{k-1}: B_t \in \prt \bT_{k-1}\}$
and let $\bT'_k = \calT_r^d + r x$, where $x\in \Z^d$ is chosen so that $B_{U_k} \in \bT'_k$ (if such an $x$ is not unique then we choose one of the $x$'s in an arbitrary way).
Let $\bT_k$ be the $d$-dimensional cube in $\R^d$ with the same center as $\bT'_k$ but with the edge length equal to $3r$.
By our assumption on the positions of $x_1$ and $y_1$, the distance from $B_{U_k}$ to $\bS(x_1,1) \cup \bS(y_1,1)$ is greater than or equal to $b-1$ for all $k$, a.s.
Suppose that $z  \notin \ol{\calB(x_1, b)}\cup \ol{\calB(y_1,b)}$. 
Let $\alpha(2,b) = 1/\log b$ and $\alpha(d,b) = b^{2-d}$ for $d\geq 3$.
The following estimate is standard, 
\begin{align*}
\P^{z}(T^B(\calS(x_1,1)) < U_2 ) 
\leq c_2 \alpha(d,b).
\end{align*}
The set $\bS(x_1,1) \cap \bT_{k-1}$ consists of $3^d$ copies of $\calS(x_1,1)$, each one of them at a distance greater than or equal to $b-1$ from $B_{U_{k-1}}$.
This, the strong Markov property and the last estimate imply that
\begin{align*}
\P^{z_1}(T^B(\bS(x_1,1) \cap \bT_{k-1}) < U_k \mid \calF_{U_{k-1}}) 
\leq c_3 \alpha(d,b).
\end{align*}
Since the $d$-dimensional Lebesgue measure of $\bA(x_1) \cap \bT_{k-1}$ is greater than $(3r)^d/2$, we have for some $c_4$,
\begin{align*}
\P^{z_1}(T^B(\bA(x_1) \cap \bT_{k-1}) < U_k \mid \calF_{U_{k-1}}) 
\geq c_4.
\end{align*}
For the same reason 
\begin{align*}
\P^{z_1}(T^B(\bA(y_1) \cap \bT_{k-1}) < U_k \mid \calF_{U_{k-1}}) 
\geq c_4,
\end{align*}
so
\begin{align*}
\P^{z_1}(T^B(\prt \bA(x_1) \cap \bT_{k-1}) < U_k \mid \calF_{U_{k-1}}) 
\geq c_4.
\end{align*}
The above estimates imply that 
\begin{align*}
\P^{z_1}(T^B(\bS(x_1,1) ) \leq T^B(\prt  \bA(x_1)) ) 
\leq c_5 \alpha(d,b).
\end{align*}
For any $x\in \prt  \bA(x_1)$, we have
$\P^x(T^B(\bS(x_1,1) )  \leq T^B(\bS(y_1,1) ) ) =1/2$, so
\begin{align*}
\P^{z_1}(T^B(\bS(x_1,1) )  \leq T^B(\bS(y_1,1) ) ) \leq 1/2 + c_5 \alpha(d,b).
\end{align*}
By analogy,
\begin{align*}
\P^{z_1}(T^B(\bS(y_1,1) )  \leq T^B(\bS(x_1,1) ) ) \leq 1/2 + c_5 \alpha(d,b).
\end{align*}
The last two estimates imply \eqref{m9.1}-\eqref{au15.1}.
\end{proof}

Recall that for an open set $D$ and a point $x\in D$, $\mu^D_x(dy)$ is the harmonic measure on $\prt D$ with the base point $x$ and
$\calU_x$ denotes the uniform probability distribution on $\calS(x,1)$.

\begin{lemma}\label{au5.2}
Suppose that $x_1, y_1 \in \calT^d_r$ and
$\dist(x_1, y_1) > 2b$.
Let $D = \calT^d_r \setminus  (\ol{\calB(x_1,1) }\cup \ol{ \calB(y_1,1)})$.
There exists $c_1$ such that for $b>4$ and $r> 8b$ there exists  $a\in(0,1]$ such that for $z  \in \calS(x_1, b)\cup \calS(y_1,b)$
there exists  a probability distribution $\calD$ on $\calS(x_1, 1)$
satisfying
\begin{align}\label{au4.2}
\mu^{D}_{z}(\,\cdot\,) / \mu^{D}_{z}(\calS(x_1,1)) = a \calU_{x_1} + (1-a) \calD,
\qquad \hbox{  on  } \calS(x_1,1),
\end{align}
and
\begin{align}\label{au4.3}
1-a < c_1 b^{1-d}.
\end{align}
\end{lemma}

Heuristically, \eqref{au4.2}-\eqref{au4.3} say that the exit distribution from $D$,
normalized so that its restriction to $\calS(x_1,1)$ is a probability measure,  is very close to the uniform distribution on $\calS(x_1,1)$.

\begin{proof}
{\it Step 1}.
By translation invariance we may and will suppose that $x_1=0$. 

Suppose that the lemma has been proved for all $z\in  \calS(0, b)$. If $z\in  \calS(y_1, b)$ then one can apply the strong Markov property at the hitting time of $ \calS(0, b)$ and use standard arguments to extend the claim to  $z\in  \calS(y_1, b)$. Hence, we will assume that $z\in  \calS(0, b)$. 

Let $D_1 =\calB(0, b) \setminus \ol{\calB(0,1)}$.
We will write $x= (x^1, x^2, \dots, x^d)$ for $x\in \R^d$.
Let $x_0= (b/2,0,\dots,0)$ and let $f(x) = \mu^{D_1}_{x_0}(dx)/\calU_{0}(dx)$ for $x\in \calS(0,1)$.
By rotational symmetry, $f(x)$ is a function of $x^1$ only.
We will show that $f(x)$ is a non-decreasing function of $x^1$.
Suppose that $x,y\in \calS(0,1)$ and $x^1 > y^1$ and let $M$ be the $(d-1)$-dimensional hyperplane such that $x$ and $y$ are symmetric with respect to $M$. Note that $M$ passes through the origin and, therefore, $D_1$ is symmetric with respect to $M$. The points $x_0$ and $x$ are on the same side of $M$. If we start a Brownian motion at $x_0$ and it hits $M$ then it has the same chance of exiting $D_1$ at $x$ and $y$, by symmetry. But Brownian motion starting at $x_0$ can exit $D_1$ at $x$ without hitting $M$, so $f(x) \geq f(y)$.

Let $D_2=\{x\in D_1: x^1 >0\}$,  $A_1 = \{x\in \calS(0,1): x^1 \geq 1/2\}$, $z_1 =(1,0, \dots, 0)$, and $h(x) = \mu^{D_2}_{x_0}(dx)/\calU_{0}(dx)$. We will argue that there exists a constant $c_2$ not depending on $b>4$, such that 
$h(y) > c_2 h(z_1)$ if $y\in A_1$.
Let $G(x_0, \,\cdot\,)$ be Green's function in $D_2$ and let $G_*(x) = 1-|x|^{2-d}$ if $d\geq 3$, and $G_*(x) = \log |x|$ if $d=2$. The functions $G(x_0, \,\cdot\,)$ and $G_*(\,\cdot\,)$ are positive and harmonic in $D_3 := \calB(0,2) \cap D_2$, and vanish continuously on $\{x\in\prt D_3: |x| =1, x^1 \geq 1/4\} $. Hence, by the boundary Harnack principle, there exists $c_3$ such that if $x, y\in  \{z\in D_3:
|z|\leq 3/2, z^1\geq 3/8\}$ then
\begin{align*}
\frac{G(x_0, x)}{G(x_0, y)} \geq c_3 \frac{G_*( x)}{G_*( y)}.
\end{align*}
If in addition $|x| = |y|$  then $G_*( x)=G_*( y)$ and we obtain
\begin{align*}
\frac{G(x_0, x)}{G(x_0, y)} \geq c_3 .
\end{align*}
Let $c_4$ be the surface area of $\calS(0,1)$.
The last estimate implies that for any $v\in A_1$,
\begin{align}\label{m8.1}
c_4 h(v) = \lim_{x\to v, x\in D_2} \frac{G(x_0, x)}{1- |x|}
\geq c_3 \lim_{y\to z_1, y\in D_2} \frac{G(x_0, y)}{1- |y|}
= c_3 c_4 h(z_1) .
\end{align}

Let  $D_4 = \{x\in \calB(0,b): x^1 >0\}$, $D_5 = \{x\in \R^d: x^1 >0\}$ and $A_2 =\calB(0,1) \cap \prt D_5 $. It is easy to  see that for some constant $c_5$ not depending on $b$ (for $b> 4$) and for all $x\in A_1$, we have $\mu^{D_4}_x(A_2) > c_5$. 
It follows from this and \eqref{m8.1} that
\begin{align*}
\mu^{D_4}_{x_0} (A_2) \geq \int _{A_1} \mu^{D_4}_x(A_2) \mu^{D_2}_{x_0} (dx)
\geq c_5 \mu^{D_2}_{x_0} (A_1)
\geq c_5 \calU_{0}(A_1) \inf_{x\in A_1} h(x)
\geq c_6 h(z_1).
\end{align*}
It is well known that $\mu^{D_5}_{x_0} (A_2) \leq c_7 b^{1-d}$ so
\begin{align*}
h(z_1) \leq c_6^{-1} \mu^{D_4}_{x_0} (A_2) 
\leq c_6^{-1} \mu^{D_5}_{x_0} (A_2)  \leq c_8 b^{1-d}.
\end{align*}

If Brownian motion starting from $x_0$ hits $\prt D_5$ before exiting $D_1$ then it can exit $D_1$ through $z_1$ and $-z_1$ with equal probabilities. Hence, $f(z_1) = h(z_1) + f(-z_1)$ and
\begin{align*}
 f(z_1) - f(-z_1) = h(z_1) \leq c_8 b^{1-d}.
\end{align*}
We have shown at the beginning of the proof that $f(-z_1) \leq f(x) \leq f(z_1) $
for all $x\in \calS(0,1)$ so for all $v_1,v_2\in \calS(0,1)$,
\begin{align}\label{s28.2}
|f(v_1) - f(v_2)| \leq c_8 b^{1-d}.
\end{align}
 
\medskip
{\it Step 2}.
Suppose that $B$ is Brownian motion on $\calT^d_r$ starting from $z \in \calS(0, b) \cup \calS(y_1, b)$ and let
\begin{align*}
R_0 & = 0,\\
V_k &= \inf\{ t\geq R_{k-1}: B_t \in  \calS(0, b/2) \cup \calS(y_1, b/2)\},
\qquad k\geq 1,\\
R_k &= \inf\{ t\geq V_{k}: B_t \in  \calS(0, b) \cup \calS(y_1, b) \cup
\calS(0, 1) \cup \calS(y_1, 1)
\},
\qquad k\geq 1,\\
K & = \inf\{k\geq 1: B_{R_k} \in \calS(0, 1) \cup \calS(y_1, 1) \}.
\end{align*}
Recall that $T^B(A)$ denotes the first hitting time of $A$ for any set $A$.
Standard methods show that for some $c_9 $ and all $x \in \calS(0, b) \cup \calS(y_1, b)$ and $y_1, y_2 \in \calS(0, b/2) \cup \calS(y_1, b/2)$,
\begin{align*}
\frac{\P^{x}(B_{T^B(\calS(0, b/2) \cup \calS(y_1, b/2))} \in dy_1)}
{\P^{x}(B_{T^B(\calS(0, b/2) \cup \calS(y_1, b/2))} \in dy_2)}
\leq c_9.
\end{align*}
Similarly, for some $c_{10} $ and all $x \in \calS(0, b/2) \cup \calS(y_1, b/2)$ and $y_1, y_2 \in \calS(0, b) \cup \calS(y_1, b)$,
\begin{align*}
\frac{\P^{x}(B_{T^B(\calS(0, b) \cup \calS(y_1, b) \cup
\calS(0, 1) \cup \calS(y_1, 1))} \in dy_1)}
{\P^{x}(B_{T^B(\calS(0, b) \cup \calS(y_1, b) \cup
\calS(0, 1) \cup \calS(y_1, 1))} \in dy_2)}
\leq c_{10}.
\end{align*}
Now the same argument that leads to \eqref{au2.11} in Step 2 of the proof of Lemma \ref{au4.1} yields existence of $c_{11}$ and $c_{12}$ such that for 
$y_1, y_2 \in \calS(0, b/2) \cup \calS(y_1, b/2)$ and $k\geq 1$,
\begin{align}\label{s28.3}
\frac{\P^{z}(B_{V_k} \in dy_1 \mid K > k-1)}
{\P^{z}(B_{V_k} \in dy_2\mid K > k-1)}
\geq 1-c_{11}e^{-c_{12} k}.
\end{align}
For $x\in \calS(0,1)$ and $k\geq 1$, let 
\begin{align*}
f_k(x) = \frac{\P^z(B_{R_k} \in dx; K> k-1) }
{\calU_0(dx)}.
\end{align*}
It follows from \eqref{s28.2} and \eqref{s28.3} that
for all $v_1,v_2\in \calS(0,1)$ and $k\geq 1$,
\begin{align*}
|f_k(v_1) - f_k(v_2)| \leq c_{13} b^{1-d}e^{-c_{12} k}.
\end{align*}
Note that $f_*(x) := \mu^D_z(dx)/\calU_0(dx) = \sum _{k\geq 1} f_k(x)$. The last estimate implies that 
for all $v_1,v_2\in \calS(0,1)$,
\begin{align*}
|f_*(v_1) - f_*(v_2)| \leq c_{14} b^{1-d}.
\end{align*}
This can be easily translated into \eqref{au4.2}-\eqref{au4.3}, taking into account the estimate for $\mu^{D}_{z}(\calS(x_1,1))$ given in Lemma \ref{au14.1}.
\end{proof}

\section{Invariance principle}\label{clt}

\begin{proof}[Proof of Theorem \ref{j26.2}]
\emph{Step 1}.
Suppose that $b>10$, $r> 10b^2$ and let $a$ be as in Lemma \ref{au5.2}. Recall that the state space for each of the processes $B, X$ and $Y$ is $\calT^d_r$. 
Assume that 
$\dist(X_0, Y_0) > b^2$, $\dist(B_0, X_0) \geq b$ and $\dist(B_0, Y_0) \geq b$.
Let
\begin{align*}
T_1 & = \inf\{ t\geq 0: B_t \in \calX_0 \cup \calY_0\},  \\
U_k &= \inf\{ t\geq T_{k}: B_t \in  \calS(X_t, b) \cup \calS(Y_t, b)\},\qquad k\geq 1, \\
T_k & = \inf\{ t\geq U_{k-1}: B_t \in \calX_t \cup \calY_t\}, \qquad k\geq 2,\\
k_1 & = \inf\{k: \dist(X_t, Y_t) < 2b 
\hbox{  for some  } t \leq U_k\},\\
j(X,1) & = \inf\{k\geq 1: B_{T_k} \in \calX_{T_k}\},\\
j(X,n) & = \inf\{k> j(X, n-1): B_{T_k} \in \calX_{T_k}\}, \qquad n \geq 2,\\
j(Y,1) & = \inf\{k\geq 1: B_{T_k} \in \calY_{T_k}\},\\
j(Y,n) & = \inf\{k> j(Y, n-1): B_{T_k} \in \calY_{T_k}\}, \qquad n \geq 2.
\end{align*}
Recall \eqref{j12.3} and let 
\begin{align*}
\bL^X_t &= \int_0^t \n( \calX_s,B_s) dL^X_s, \qquad
\bL^Y_t = \int_0^t \n( \calY_s,B_s) dL^Y_s,\\
\Delta_n \bL^X &= \bL^X_{U_{j(X,n)}} -\bL^X_{T_{j(X,n)}},\qquad
\Delta_n \bL^Y = \bL^Y_{U_{j(Y,n)}} -\bL^Y_{T_{j(Y,n)}}.
\end{align*}
Note that for $k< k_1$, $X_{U_k}  = X_0 + \sum _{1\leq n\leq k} \Delta_n \bL^X$,
and a similar formula holds for $Y$.

The strong Markov property and 
Lemma \ref{au5.2} imply that for $n\geq 1$ such that $j(X,n) -1 < k_1$, the conditional distribution of $B_{T_{j(X,n)}}$ given $\calF_{U_{j(X,n) -1}}$
is equal to $a \calU_{X_{U_{j(X,n) -1}}} + (1-a) \calD^X_n$, where $\calD^X_n$ is a probability distribution on $\calX_{U_{j(X,n) -1}}$, determined by the values of
$X_{U_{j(X,n) -1}}, Y_{U_{j(X,n) -1}}$ and $B_{U_{j(X,n) -1}}$.
Recall formulas \eqref{m22.1}-\eqref{au5.3} and for $x\in \calS(0,1)$ let $\calD^b_x$ be the distribution of $\bL^Z_{T^Z_{\calS(0,b)}}$, assuming that $B_0=x$.
Let $\calD^b = \int_{\calS(0,1) } \calD^b_x \calU_0(dx)$ and let $\{\bM_n^1\}_{n\geq 1}$ be a sequence of i.i.d. random vectors with distributions $\calD^b$.
Let $\{\delta_n\}_{n\geq 1}$ be i.i.d. random variables with $\P(\delta_n = 1) = 1-\P(\delta_n=0) = a$. We assume that $\{\delta_n\}_{n\geq 1}$ and $\{\bM_n^1\}_{n\geq 1}$ are independent.
We will define another process $\{\bM_n^2\}_{n\geq 1}$. Before doing so, we note that
we can and will assume that $\{\bM_n^1\}_{n\geq 1}$, $\{\bM_n^2\}_{n\geq 1}$ and $\{\delta_n\}_{n\geq 1}$ are defined on the same probability space as $(B,X,Y)$.
For every $n\geq 1$, let $\bM_n^2$ be a random vector with the conditional distribution $\int_{\calS(0,1) } \calD^b_x \calD^X_n(dx)$ given $\calF_{U_{j(X,n) -1}}$. 
Let $\bM_n = \delta_n \bM^1_n + (1-\delta_n) \bM_n^2$ for $n\geq 1$.
It is elementary to check that the sequences $\{\bM_n\}_{n\geq 1, j(X,n) -1 < k_1}$ and $\{\Delta_n \bL^X\}_{n\geq 1, j(X,n) -1 < k_1}$ have the same distributions.

Recall $\lambda_1(d,b)$ from \eqref{au9.1}.
Let $\{N^b_t, t\geq 0\}$ be a Poisson process with the rate
(expected number of jumps per unit of time) equal to $ \lambda_3(d,b):=
b^4 \lambda_1(d,b)$.
We assume that $N^b$ is independent of $\{\delta_n \bM^1_n\}_{n\geq 1}$.
Let $\{\bN ^1_t, t\geq 0\}$ be a continuous time pure jump process with values in $\R^d$, starting from 0, with jump times matching those of $N^b$. For $n\geq 1$, the $n$-th jump of $\bN^1$ is equal to $\delta_n \bM^1_n/b^2$.

Let $\bPi$ be the covariance matrix equal to  the unit diagonal matrix times $ 2 /( (d-1 ) d) $. 
We will use the invariance principle in the form given in \cite[Thm. IX 4.21]{JS} to show that the processes $\{\bN^1_t, t\geq 0\}$ converge weakly to Brownian motion with the covariance matrix $\bPi$ as $b\to \infty$. 
To apply \cite[Thm. IX 4.21]{JS}, one needs to check three conditions. 
Their condition (iii) is concerned with the initial distributions and it is clearly satisfied in our case---the initial distributions converge to the delta function at 0. Condition (i) is an assumption on the asymptotic form of the expectation and variance of the jumps. The jumps of $\bN^1$ are symmetric so the expected value of the jumps is zero. 
Note that the first coordinate $(\bM^1_n)_1  $ of  $\bM^1_n  $ has the same distribution as $ \bL^{1,Z}_{T_b}$ in Lemma 
\ref{a3.1}.
Hence, by Lemma \ref{a3.1} (i)-(ii) and \eqref{au4.3}, the variance of the first component of the limit is equal to
\begin{align*}
&\lim_{b\to \infty} b^4 \lambda_1(d,b) \E(\delta_1 (\bM^1_n)_1/b^2)^2
= \lim_{b\to \infty}  \lambda_1(d,b) a \E( (\bM^1_n)_1 )^2 \\
& 
=  \lim_{b\to \infty}  \lambda_1(d,b) a \E( \bL^{1,Z}_\infty )^2 
= \lim_{b\to \infty}  \lambda_1(d,b) a \lambda_2(d,b)
= \frac{2 }{ (d-1 ) d}.
\end{align*}
The covariance structure of the limit is represented by a constant multiple of the unit diagonal matrix because of the rotational symmetry of $\bM^1_n$. Finally, the Lindeberg-Feller-type condition (ii) in \cite[Thm. IX 4.21]{JS} has been verified in Lemma \ref{a3.1} (iii). We conclude that  processes $\{\bN^1_t, t\geq 0\}$ converge weakly to Brownian motion with the covariance matrix $\bPi$.

We define $\{\wt \bM_n^1\}_{n\geq 1}$, $\{\wt\bM_n^2\}_{n\geq 1}$, $\{\wt \delta_n\}_{n\geq 1}$ and $\{\wt\bM_n\}_{n\geq 1}$ relative to $Y$
in the same way as
$\{\bM_n^1\}_{n\geq 1}$, $\{\bM_n^2\}_{n\geq 1}$, $\{\delta_n\}_{n\geq 1}$ and $\{\bM_n\}_{n\geq 1}$ were defined relative to $X$.
We can and will assume that $\{\wt \bM_n^1\}_{n\geq 1}$ and $\{\wt \delta_n\}_{n\geq 1}$ are independent of $\{ \bM_n^1\}_{n\geq 1}$ and $\{ \delta_n\}_{n\geq 1}$.
We assume that all processes 
$\{\bM_n^1\}_{n\geq 1}$, $\{\bM_n^2\}_{n\geq 1}$, $\{\delta_n\}_{n\geq 1}$, $\{\bM_n\}_{n\geq 1}$, $\{\wt \bM_n^1\}_{n\geq 1}$, $\{\wt\bM_n^2\}_{n\geq 1}$, $\{\wt \delta_n\}_{n\geq 1}$ and $\{\wt\bM_n\}_{n\geq 1}$  are defined on the same probability space as $(B,X,Y)$.

Recall the process $N^b$.
We assume that $N^b$ is independent of $\{\wt\delta_n \wt\bM^1_n\}_{n\geq 1}$.
Let $\{\wt \bN ^1_t, t\geq 0\}$ be constructed from $\{\wt\delta_n \wt\bM^1_n\}_{n\geq 1}$ and $N^b$ in the same way as $\{ \bN ^1_t, t\geq 0\}$ was constructed from $\{\delta_n \bM^1_n\}_{n\geq 1}$ and $N^b$. Note that we use the same Poisson process $N^b$ in both cases. 

Let $\{(W^X_t, W^Y_t), t\geq 0\}$ be a pair of independent $d$-dimensional Brownian motions, each with variance $ 2/ ( (d-1 ) d) $. It follows from independence of $\{\wt\delta_n \wt\bM^1_n\}_{n\geq 1}$ and $\{\delta_n \bM^1_n\}_{n\geq 1}$ that $\{(\bN^1_t, \wt \bN^1_t), t\geq 0\}$ converge weakly to
$\{(W^X_t, W^Y_t), t\geq 0\}$.

Let $\{\bN ^2_t, t\geq 0\}$ be a continuous time pure jump process with values in $\R^d$, starting from 0, with jump times matching those of $N^b$. For $n\geq 1$, the $n$-th jump of $\bN^2$ is equal to $(1-\delta_n) \bM^2_n/b^2$.
 We obtain from \eqref{a8.1} that for $b>2$ and $d\geq 2$,
\begin{align*}
\E|\bM^2_n|^2\leq d^2(2 + 2 c_1 \lambda_2(d,b)) \leq c_2 +  c_3 \log b.
\end{align*}
According to \eqref{au4.3}, $\P( 1-\delta_k \ne 0) \leq c_4 b^{1-d}$. It follows that, for large $b$,
\begin{align*}
\E\left| (1-\delta_n) \bM^2_n/b^2\right|^2 \leq 
 c_5 b^{1-d} 
(c_2 +  c_3 \log b) /b^4  )
\leq c_6 b^{-3-d}\log b.
\end{align*}
Recall that $ \lambda_3(d,b)=
b^4 \lambda_1(d,b)$.
Let $n(b) = \lceil \lambda_3(d,b) \rceil $ and note that $n(b) \leq 2 (d-2) b^4 $ for large $b$. It is easy to see that  $\{|\bM_n^2|\}_{n\geq 1}$ are i.i.d. By Doob's maximal inequality,
\begin{align*}
\P&\left(\sup_{1 \leq n \leq n(b)} \sum_1^n (1-\delta_n) |\bM^2_n|/b^2 \geq 
b^{-1/2}\right)
\leq b^{1/2} n(b) \E\left| (1-\delta_n) \bM^2_n/b^2\right|^2\\
&\leq 2 (d-2) b^{9/2} c_6 b^{-3-d}\log b
= c_7 b^{3/2-d} \log b \leq c_7 b^{-1/2} \log b
\xrightarrow{b\to \infty} 0.
\end{align*}

Note that $\bN^2$ will have about $n(b)$ jumps by time 1. Now
standard arguments show that $\{\bN^2_t, t\in[0,1]\}$ converge to the process identically equal to 0 as $b\to \infty$. A similar argument shows that for every fixed $t_1 < \infty$,
$\{\bN^2_t, t\in[0,t_1]\}$ converge to the process identically equal to 0 as $b\to \infty$. 

Let $\{\wt \bN ^2_t, t\geq 0\}$ be defined relative to $\{(1-\wt\delta_n) \wt\bM^2_n\}_{n\geq 1}$ in the same way as $\{ \bN ^2_t, t\geq 0\}$ was defined relative to $\{(1-\delta_n) \bM^2_n\}_{n\geq 1}$. In both cases we use the same Poisson process $N^b$. By analogy, for every fixed $t_1 < \infty$, $\{\wt \bN^2_t, t\in[0,t_1]\}$ converge to the process identically equal to 0 as $b\to \infty$.

Let $\bN_t = \bN^1_t + \bN^2_t$ and $\wt\bN_t = \wt\bN^1_t + \wt\bN^2_t$. Combining the results on convergence of $\{(\bN^1_t, \wt \bN^1_t), t\geq 0\}$, $\{\bN^2_t, t\in[0,t_1]\}$ and $\{\wt\bN^2_t, t\in[0,t_1]\}$, we see that $\{(\bN_t, \wt \bN_t), t\geq 0\}$ converge weakly to
$\{(W^X_t, W^Y_t), t\geq 0\}$ as $b\to \infty$.

Recall the definition of $T_k$ from the beginning of the proof
and let 
\begin{align*}
\wh N^{X,b}_t  &= \inf \{ k: \sigma^X_{T_k} \geq t/b^4\}, \quad
\wh N^{Y,b}_t  = \inf \{ k: \sigma^Y_{T_k} \geq t/b^4\},\quad
\wh N^b_t  = \inf \{ k: \sigma_{T_k} \geq t/b^4\}.
\end{align*}

It follows from \eqref{m9.3}, \eqref{a3.2}, \eqref{au9.1} and the definition of $N^b$ that $\{\wh N^b_t, t\geq 0\}$ and $\{ N^b_t, t\geq 0\}$ have the same distributions. It is standard to prove that for every $t_1< \infty$, a.s.,
\begin{align}\label{au10.1}
\lim_{b\to \infty}
\sup_{0 \leq t \leq t_1}
\frac{ |\wh N^b_t - N^b_t|}{\lambda_3(d,b)} = 0.
\end{align}

Note that $\wh N^{X,b} + \wh N^{Y,b} = \wh N^b$.
By Lemma \ref{au14.1}, for any arbitrarily small $\eps>0$ there exists $b_1$ such that for $b> b_1$ and all $n\geq 1$,  the time of the $n$-th jump of $\wh N^b$ is equal to a jump time of $\wh N^{X,b}$ with probability in the range $(1/2-\eps, 1/2+\eps)$. This holds conditional on the times of jumps of $\wh N^{X,b} , \wh N^{Y,b} $ and $ \wh N^b$ before the time of the $n$-th jump of $\wh N^b$.
These observations and \eqref{au10.1} imply easily that for every $t_1< \infty$, a.s.,
\begin{align}\label{au11.2}
\lim_{b\to \infty}
\sup_{0 \leq t \leq t_1}
\frac{ |\wh N^{X,b}_t - N^b_t/2|}{\lambda_3(d,b)} = 
\lim_{b\to \infty}
\sup_{0 \leq t \leq t_1}
\frac{ |\wh N^{Y,b}_t - N^b_t/2|}{\lambda_3(d,b)} =0.
\end{align}

Let $\{\bR_t, t\geq 0\}$ be a pure jump process with the same jumps as those of the process $\{\bN_t, t\geq 0\}$ except that their times are determined by the jumps of $\wh N^{X,b}$ rather than $N^b$. 
We define $\{\wt\bR_t, t\geq 0\}$ in a similar way relative to $\{\wt\bN_t, t\geq 0\}$ and $\wh N^{Y,b}$.
 It follows from convergence of $\{(\bN_t, \wt \bN_t), t\geq 0\}$ and \eqref{au11.2} that $\{(\bR_t, \wt \bR_t), t\geq 0\}$ converge weakly to
$\{\sqrt{2}(W^X_t, W^Y_t), t\geq 0\}$ as $b\to \infty$.

Recall that the sequences $\{\bM_n\}_{n\geq 1, j(X,n) -1 < k_1}$ and $\{\Delta_n \bL^X\}_{n\geq 1, j(X,n) -1 < k_1}$ have the same distributions.
It follows that we could construct copies of $\{\bR_t, t\geq 0\}$
and $\{b^{-2}\bL^X( \sigma^X_{b^4 t}), t\geq 0\}$ on the same probability space so
that $\bR_{T_{j(X,n)}} = b^{-2}\bL^X( \sigma^X_{b^4 T_{j(X,n)}})$ for $n\geq 1, j(X,n) -1 < k_1$. The process $\wh N^{X,b}$ has about $\lambda_3(d,b)$  jumps per unit of time.
We bound the difference between the processes $\bR_t$ and $b^{-2}\bL^X( \sigma^X_{b^4 t})$ as follows.
\begin{align*}
\sup_{0 \leq t \leq U_{\lceil \lambda_3(d,b) \rceil \land k_1}} |\bR_t - b^{-2}\bL^X( \sigma^X_{b^4 t})|
&\leq \sup _{1\leq n\leq \lceil \lambda_3(d,b) \rceil}
\sup_{T_{j(X,n)} \leq t \leq U_{j(X,n)}}
b^{-2}|\bL^X_{t} -\bL^X_{T_{j(X,n)}}|\\
&\leq \sup _{1\leq n\leq \lceil \lambda_3(d,b) \rceil}
\sup_{T_{j(X,n)} \leq t \leq U_{j(X,n)}}
b^{-2}|L^X_{t} -L^X_{T_{j(X,n)}}|\\
&= \sup _{1\leq n\leq \lceil \lambda_3(d,b) \rceil}
b^{-2}(L^X_{U_{j(X,n)}} -L^X_{T_{j(X,n)}}).
\end{align*}
By \eqref{m9.3}, \eqref{a3.2} and \eqref{au9.1}, the distribution of $L^X_{U_{j(X,n)}} -L^X_{T_{j(X,n)}}$ is exponential with mean $1/ \lambda_1(d,b)$. Hence,
\begin{align*}
\P&\left(\sup _{1\leq n\leq \lceil \lambda_3(d,b) \rceil}
b^{-2}(L^X_{U_{j(X,n)}} -L^X_{T_{j(X,n)}})
\geq b^{-1/2}\right)\\
&\leq
\lceil \lambda_3(d,b) \rceil
\P\left(L^X_{U_{j(X,n)}} -L^X_{T_{j(X,n)}}
\geq b^{3/2}\right)
 = \lceil \lambda_3(d,b) \rceil \exp( - \lambda_1(d,b) b^{3/2}).
\end{align*}
The last quantity goes to 0 as
$b \to \infty$ so
\begin{align*}
\lim_{b\to \infty}
\P&\left(\sup_{0 \leq t \leq U_{\lceil \lambda_3(d,b) \rceil \land k_1}} |\bR_t - b^{-2}\bL^X( \sigma^X_{b^4 t})|
\geq b^{-1/2}\right)=0.
\end{align*}
For the same reason we have
\begin{align*}
\lim_{b\to \infty}
\P&\left(\sup_{0 \leq t \leq U_{\lceil \lambda_3(d,b) \rceil \land k_1}} |\wt\bR_t - b^{-2}\bL^Y( \sigma^Y_{b^4 t})|
\geq b^{-1/2}\right)=0.
\end{align*}
Note that the last two formulas still hold if we replace $\lambda_3(d,b)$ with any constant multiple of $\lambda_3(d,b)$. This and 
the weak convergence of 
$\{(\bR_t, \wt \bR_t), t\geq 0\}$  to
$\{\sqrt{2}(W^X_t, W^Y_t), t\geq 0\}$ 
imply that $\{(b^{-2}\bL^X( \sigma^X_{b^4 t}),b^{-2}\bL^Y( \sigma^Y_{b^4 t})), t\geq 0\}$ converge to $\{\sqrt{2}(W^X_t, W^Y_t), t\geq 0\}$ 
as $b\to \infty$. 

\medskip

We will now discuss a few technical points that were partly
swept under the rug in the proof so far.
First, we have just made a claim of convergence of some processes
on the half-line although the construction of stopping times used
in the proof stops at $U_{k_1}$. We assumed that $\dist(X_0, Y_0) > b^2$. At time $U_{k_1}$, the processes $X$ and $Y$ are about $2b$ units apart. After rescaling by $b^{-2}$, this corresponds to the time when $b^{-2}X$ and $b^{-2}Y$ starting at a distance greater than $1$ come closer than $2/b$ units apart. Since $d$-dimensional Brownian motion does not hit a fixed point for $d\geq 2$, this time goes to infinity in probability as $b\to \infty$.
This justifies the assertion that convergence holds on the whole time half-line $[0,\infty)$.

We can drop the assumption that $\dist(B_0, X_0) \geq b$ and $\dist(B_0, Y_0) \geq b$
as follows. The assumption is satisfied at the time $U_1$ so the invariance principle holds for the post-$U_1$ process. The amount of local time and the maximal displacement of the processes on the interval $[0, U_1]$ can be easily estimated using the same methods that were used to estimate $\Delta_n \bL^X$. The estimates show that the initial part of the process, on the time interval $[0, U_1]$, will disappear in the limit of rescaled processes.

We can now change the clocks from $\sigma^X_t$ and $\sigma^Y_t$ to the common clock $\sigma_t$ due to \eqref{au11.2}. We conclude that 
$\{(b^{-2}\bL^X( \sigma_{b^4 t}),b^{-2}\bL^Y( \sigma_{b^4 t})), t\geq 0\}$ converge to $\{(W^X_t, W^Y_t), t\geq 0\}$ 
as $b\to \infty$. It is straightforward to check that this implies the theorem under the assumption that for each $n$,  $\dist(X_0, Y_0) > n^{-1/2}$. 

We note that the same proof would apply if for some fixed $c_8>0$ and all $n$ we assumed that $\dist(X_0, Y_0) > c_8 n^{-1/2}$.

We also note that our estimates are uniform in the sense that they do not depend on the initial positions of $X,Y$ and $B$. We will make this claim more precise. Recall that the Prokhorov metric is a way to metrize weak convergence. For every $T,\eps, c_9>0$ there exists $n_1$ such that for all $n\geq n_1$, all $x_1, y_1$ and $z_1$ such that $\dist(x_1, y_1) \geq c_9 n^{-1/2}$, $\dist(x_1, z_1)\geq 1$ and $\dist(y_1, z_1)\geq 1$, if $X_0=x_1$, $Y_0=y_1$ and $B_0=z_1$ then 
the Prokhorov distance between
 $\{C_d\, n^{-1/2} (\fX^r_{ \sigma_{nt}} -\fX^r_0, \fY^r_{\sigma_{nt}} - \fY^r_0), t\in [0,T] \}$ and standard $(2d)$-dimensional Brownian motion on $[0,T]$ is less than $\eps$.

\medskip
\emph{Step 2}. 
We will show that for any $p_0<1$ and $t_0>0$ there exist $n_0$ and $\gamma>0$ such that for any $n\geq n_0$ and any starting point $(B_0, X_0, Y_0)$ satisfying the usual conditions $|B_0 - X_0| \geq 1$, $|B_0 - Y_0| \geq 1$ and $|X_0 - Y_0| \geq 2$, the process 
$|n^{-1/2} (\fX_{ \sigma_{nt}} - \fY_{\sigma_{nt}})|$ will become greater than $\gamma$ in at most $t_0$ units of time with probability greater then $p_0$.

Let
\begin{align*}
A_t &=\fX_{ \sigma_t} - \fY_{\sigma_t},\\
T^1_k &= \inf\{t\geq 0: A_t \notin \calB(0, 2^k)\},\\
T^2_k &= \inf\{t\geq 0: A_t \notin \calB(0, 2^k) \setminus \calB(0, 2^{k-2})\}.
\end{align*}
The process $A_t$ is not Markovian but the process $(A_t, X_{ \sigma_t} - B_{\sigma_t})$ is. We will write $\P^{x,y}$ to denote the distribution of  
$\{(A_t, X_{ \sigma_t} - B_{\sigma_t}),t\geq0\}$ starting from $(A_0, X_{ \sigma_0} - B_{\sigma_0})=(x,y)$.
The last remark in Step 1 and standard Brownian estimates show that
there exist $p_1,p_2>0$ and $k_2 $ such that for $k \geq k_2$ and $|y|\geq 1$,
\begin{align}\label{a10.1}
&\P^{x,y}( T^1_{k} =T^2_{k}) \geq p_1 \qquad \hbox{ for } 
x \notin\calB(0, 2^{k-1}),\\
&\P^{x,y}(T^2_{k} \leq 2^{2k} ) \geq p_2
\qquad \hbox{ for } 
x \in \calB(0, 2^k) \setminus \calB(0, 2^{k-2}). \label{a10.2}
\end{align}

Note that we can take $p_1$ to be any number less than $1/2$ for any $d$, so we will assume that $p_1 = 3/8$. 

By applying the Markov property at times $j 2^{2k}$, $j=1,2,\dots$, and
\eqref{a10.2}, we see that for some $c_{10}$ and $k\geq k_2$,
\begin{align}\label{a10.3}
\E^{x,y}T^2_{k} \leq c_{10} 2^{2k}\qquad \hbox{ for } 
x \in \calB(0, 2^k) \setminus \calB(0, 2^{k-2}).
\end{align}
We will show that there exists $c_{11} >0$ such that
for all $k \geq 0$, $x\in \ol{\calB(0, 2^k)}$ and $|y|\geq 1$, 
\begin{align}\label{a10.4}
\E^{x,y}T^1_{k} \leq c_{11} 2^{2 k}.
\end{align}
The proof will be based on induction.
 For all $k$, let
\begin{align*}
T^3_{k} = T^1_{k} \bone_{\{ T^1_{k} =T^2_{k}\}}
+ T^1_{k-1}\circ \theta_{T^2_{k}} \bone_{\{ T^1_{k} \ne T^2_{k}\}}
= T^2_{k} \bone_{\{ T^1_{k} =T^2_{k}\}}
+ T^1_{k-1}\circ \theta_{T^2_{k}} \bone_{\{ T^1_{k} \ne T^2_{k}\}},
\end{align*}
where $\theta$ denotes the usual Markovian shift operator.
Suppose that \eqref{a10.4} holds for $k_2,k_2+1, \dots, k-1$ (the value of $c_{11}$ will be specified later). In particular, we assume that \eqref{a10.4} holds for $k-1$ and
$x\in \calS(0,2^{k-1})$.
Then, by \eqref{a10.3} and \eqref{a10.4},
\begin{align}\label{a10.5}
\E^{x,y} T^3_{k} \leq c_{10} 2^{2k} + c_{11} 2^{2( k-1)}.
\end{align}
Let
\begin{align*}
T^4_1 &= T^3_{k},\\
T^4_j &= T^3_{k} \circ \theta_{ T^4_{j-1}}, \qquad j\geq 2,\\
K&= \min\{j: T^4_j = T^1_{k}\}.
\end{align*}
The distribution of $K$ is majorized by the geometric distribution with mean $1/p_1$, by \eqref{a10.1}. This, the strong Markov property applied at $T^4_j$'s and \eqref{a10.5} imply that
\begin{align}\label{a10.8}
\E^{x,y} T^1_{k} \leq (c_{10} 2^{2k} + c_{11} 2^{2 (k-1)})/p_1.
\end{align}
To complete the inductive step, we need to find $c_{11}$  such that
the last expression is less than or equal to $c_{11} 2^{2k}$. In other words, we want to have
\begin{align}\label{a10.7}
(c_{10} 2^{2k} + c_{11} 2^{2(k-1)})/p_1 \leq c_{11} 2^{2k}.
\end{align}
The following inequality is equivalent,
\begin{align}\label{a10.6}
c_{11}( 1 -1/(4p_1)) - (c_{10}/p_1) \geq 0.
\end{align}
Since $p_1 = 3/8$, we can choose $c_{11}$ so large that  \eqref{a10.6} and, therefore, \eqref{a10.7} hold. We combine this with \eqref{a10.8} to conclude that
$\E^{x,y} T^1_{k} \leq  c_{11} 2^{2k}$ which concludes the inductive step.

To initialize the inductive proof, it suffices to show that \eqref{a10.4} holds for $k=k_2$ (then \eqref{a10.4} holds for all $k\leq k_2$ with $c_{11}$ replaced by
$c_{11} 2^{2k_2}$). We only sketch the proof. If $|A_0| \leq 2^{k_2}$, it is easy to construct a deterministic smooth trajectory such that if we use it as the driving path in place of $B_t$ then $|A_t|$ will exceed $ 2^{k_2 +1}$ in no more than $ 2^{2k_2}$ units of time. By the support theorem (see \cite[Thm.~I.6.6]{Bassbook}) and
the continuity of the  Skorokhod map (see \cite[Thm.~1.1]{LS}),
with probability $p_3>0$ not depending on the starting point, 
if the driving process $B_t$ is Brownian motion then $|A_t|$ will exceed $ 2^{k_2 }$ in no more than $ 2^{2k_2+1}$ units of time.
Applying the Markov property at times $j 2^{2k_2+1}$, $j\geq 1$, we conclude that the expected value of the time when $|A_t|$  exceeds $ 2^{k_2 }$ is bounded by $2^{2k_2+1}/p_3$. This implies \eqref{a10.4} for $k=k_2$ (but we may have to enlarge $c_{11}$).

Recall that we fixed a $p_0<1$ at the beginning of the proof.
It follows from \eqref{a10.4} that 
for all $k \geq 0$, $x\in \ol{\calB(0, 2^k)}$ and $|y|\geq 1$, 
\begin{align*}
\P^{x,y} (T^1_{k} \geq c_{11} 2^{2 k}/(1-p_0)) \leq 1-p_0.
\end{align*}
By scaling, the claim made at the very beginning of Step 2 follows if we take $\gamma=((1-p_0)t_0/c_{11})^{1/2}$.

\medskip

It remains to combine the claims proved in Steps 1 and 2. According to Step 2, irrelevant of the starting position of $X,Y$ and $B$, the process $|n^{-1/2} (\fX_{ \sigma_{nt}} - \fY_{\sigma_{nt}})|$ will reach a small fixed distance in a small fixed time. After that time, we use the invariance 
principle in the form proved at the end of Step 1.
\end{proof}

\section{Irreducibility}
\label{sec:irreducibility}

The argument presented in this section is a straightforward
adaptation of the proof of \cite[Thm.~6.1]{BBCH} so we will omit many details.

 Let
$\P^{z,x,y}$ denote the distribution of $(B_t, X_t, Y_t)$ 
starting from $(z,x,y)$. 

\begin{lemma}\label{jl28.2}
Fix any $d\geq 2$ and $r>10$.
There exists a positive measure $\mu$ on $(\calT^d_r)^3$ and $t_0>0$ such that if $\mu (\Gamma) > 0$, then for all
$(z,x,y) \in (\calT^d_r)^3$
such that $|z-x| \geq 1$, $|z-y| \geq 1$ and $|x-y| \geq 2$, we have $\P^{z,x,y} ((B_{t_0},X_{t_0},Y_{t_0}) \in \Gamma)
> 0$.
\end{lemma}

\begin{proof}

Suppose that we replace Brownian motion $B$ with a continuous function $\{A_t, t\geq 0\}  $ in \eqref{m9.2}-\eqref{j12.2}. These equations have solutions according to \cite{LS}. Let $(A_t, X^A_t, Y^A_t)$ be the resulting triplet of processes. We proved in Section \ref{prelim1} that the processes are defined until the accumulation time of visits of $A$ to the unit spheres centered at $X^A$ and $Y^A$. We will consider only functions $A$ such that there is no such finite accumulation time.

Fix some $u_1, x_1, y_1 \in \calT^d_r$ and assume that for every pair of these points, the distance between them is greater than 5.
Consider any $A_0, X_0, Y_0 \in \calT^d_r$ with $|X_0 - Y_0| \geq 2$, $|A_0 - X_0| \geq 1$ and $|A_0 - Y_0| \geq 1$. 
Let $\alpha = 1/(20 d)$.
It is elementary to see that one can find a continuous function $\{A_t, t\geq 0\}  $ and a time $t_1< \infty$ not depending on $u_1,x_1,y_1, A_0, X_0, Y_0$ (but possibly depending on $d$ and $r$) such that there exists $t_2 \leq t_1$ with the property that $A_{t_2} \in \calB(u_1, \alpha)$, $X^A_{t_2} \in \calB(x_1, \alpha)$ and $Y^A_{t_2} \in \calB(y_1, \alpha)$. We briefly justify this claim.
If the spheres $\calX_0$ and $\calY_0$ touch or are very close to each other then
the function $A$ has to start by ``pushing them apart.'' Then $A$ has to push the spheres in the right direction, one at a time.
By the continuity of the  Skorokhod map (see \cite[Thm.~1.1]{LS}), there exists $\eps_1>0$ such that  if a continuous function $C_t$ satisfies $|A_t - C_t| \leq \eps_1$ for all $t\in [0,t_1]$, then 
$|(A_t, X^A_t, Y^A_t) - (C_t, X^C_t , Y^C_t)| \leq \alpha$ for $t\in [0,t_1]$.
The support theorem (see \cite[Thm.~I.6.6]{Bassbook}) implies that for any continuous function $\{A_t, t\geq 0\}  $, if $B_0 = A_0$ then  $P(\sup_{0\leq t \leq t_1}|B_t - A_t| < \eps_1)>0$. 
We conclude that $P(|(A_t, X^A_t, Y^A_t) - (B_t, X_t , Y_t)| \leq \alpha)>0$ and, therefore, if $B_0, X_0, Y_0 \in \calT^d_r$ with $|X_0 - Y_0| \geq 2$, $|B_0 - X_0| \geq 1$ and $|B_0 - Y_0| \geq 1$ then $B_{t_2} \in \calB(u_1, 2\alpha)$, $X_{t_2} \in \calB(x_1, 2\alpha)$ and $Y_{t_2} \in \calB(y_1, 2\alpha)$ with positive probability.  It is easy to see that the last claim implies that
\begin{align}\label{jl28.1}
\P(B_{t_1} \in \calB(u_1, 3\alpha), X_{t_1} \in \calB(z_1, 3\alpha), Y_{t_1} \in \calB(y_1, 3\alpha) ) >0.
\end{align}

Let
$\angle(v,w)$ denote the angle between vectors $v$ and $w$ and 
recall that  $\be_k$ is the $k$-th vector in the usual orthonormal basis for $\R^d$.
Let
 $C^j(\delta_0) = \{v
\in \R^d: \angle(\be_j, v) \leq \delta_0\}$. Fix
$\delta_0>0$  so small that for any $v_j \in C^j(2\delta_0)$, $j=1, \dots,
d$, the vectors $\{v_j\}$ are linearly independent. 
Let $C^{j,X}_t = X_t + (C^j(\delta_0) \cap \calX_t)$; this set is a small spherical cap on $\calX_t$, with center in the direction $\be_j$ from $X_t$.

Let $F_X$ be the event that all of the following conditions hold: (i) Brownian motion $B$ visits the (random and time dependent) sets
 $C^{j,X}_t$, $j=1,2, \dots , d$, in this order, between times $t_1$ and $2t_1$; (ii) $B$ does not visit any other part of $\calX_t \cup \calY_t$ during $[t_1, 2 t_1]$; (iii) the local time $L^X$ increases less  than $1/(2d)$ when $B$ is hitting $C^{j,X}_t$ during $[t_1, 2 t_1]$, for each $j=1,2, \dots , d$.

We define $C^{j,Y}_t$ and $F_Y$ in an analogous way except that $B$ 
is required to visit $C^{j,Y}_t$'s during $[3t_1, 4t_1]$.

Let $F_B$ be the event that all of the following conditions are satisfied: (i) $B$ hits $\calB(u_1,1)$ between $4t_1$ and $5t_1$; (ii) $B$ does not hit $\calX \cup \calY$ between 
the last visit to $C^{d,X}_t$ during $[t_1, 2t_2]$ and the first visit to $C^{1,Y}_t$ during $[3t_2, 4t_2]$; (iii) $B$ does not visit $\calX_t \cup \calY_t$ between 
the last visit to $C^{d,Y}_t$ during $[3t_2, 4t_2]$ and hitting of $\calB(u_1, 1)$.  

The probability of $F_X \cap F_Y \cap F_B$ is strictly positive due to the support theorem and excursion theory.

Let 
\begin{align*}
K^{j,X}& =   \int_{t_1}^{2t_1} \n(\calX_s, B_s)\bone_{\{B_s \in C^{j,X}_s\}} dL^X_s, 
\qquad L^{j,X} =   \int_{t_1}^{2t_1} \bone_{\{B_s \in C^{j,X}_s\}} dL^X_s,\\
K^{j,Y} &=   \int_{3t_1}^{4t_1} \n(\calY_s, B_s)\bone_{\{B_s \in C^{j,Y}_s\}} dL^Y_s,
\qquad L^{j,Y} =   \int_{3t_1}^{4t_1} \bone_{\{B_s \in C^{j,Y}_s\}} dL^X_s,
\end{align*}
and note that $K^{j,X}, K^{j,Y} \in C^j(\delta_0)$ for all $j=1, \dots, d$.

The components of the random vector 
\begin{align*}
\bK := (K^{1,X}, \dots, K^{d,X}, K^{1,Y}, \dots, K^{d,Y})
\end{align*}
are not independent but the fact that $F_X \cap F_Y \cap F_B$ has a positive probability and the excursion theory based argument given in the proof of \cite[Thm.~6.1]{BBCH}
show that the distribution of 
\begin{align*}
 (L^{1,X}, \dots, L^{d,X}, L^{1,Y}, \dots, L^{d,Y})
\end{align*}
 has a component with a density strictly positive on $(0, 1/(2d))^{2d}$.
For $0 \leq a_j < b_j$, $j = 1,2, \dots, 2d$, 
let 
 $\Lambda([a_1,b_1], [a_2, b_2], \dots, [a_{2d},b_{2d}])
 $
be the set of all possible values of
$K^{1,X}+ \dots+ K^{d,X}+ K^{1,Y}+ \dots+ K^{d,Y}$
assuming that $L^{j,X} \in [a_j, b_j]$ and $L^{j,Y} \in [a_{j+d}, b_{j+d}]$.
It is easy to show using the definition of $C^j(\delta_0)$'s that the
$2d$-dimensional volume of $\Lambda([a_1,b_1], \dots,
[a_{2d},b_{2d}])$ is bounded below by $c_1 \prod_{1\leq k \leq 2d}
(b_k - a_k)$, and bounded above by $c_2 \prod_{1\leq k \leq 2d}
(b_k - a_k)$. This implies that the distribution of $\bK$ has
a component with a strictly positive density on 
$\Lambda([0, 1/(2d)], [0, 1/(2d)], \dots, [0, 1/(2d)])$.
Moreover, the claim holds conditional on the sigma field $\calF_{t_1}$. This and \eqref{jl28.1} imply that the distribution of $(B_{5t_1}, X_{5t_1}, Y_{5t_1})$ has a component with a strictly positive density on 
$\calB(u_1, 1/(4d)) \times \calB(x_1, 1/(4d)) \times \calB(y_1, 1/(4d)) $.
\end{proof}

\begin{proof}[Proof of Theorem \ref{j26.3} (i)] If there were more than one invariant measure, at
least two of them (say, $\mu$ and $\nu$) would be mutually
singular by Birkhoff's ergodic theorem \cite{Sin}. However,
we have shown in Lemma \ref{jl28.2} that there exists a strictly positive
measure $\psi$ which is absolutely continuous with respect to
any transition probability, so that in particular, $\psi \ll
\mu$ and $\psi \ll \nu$. Since $\mu \perp \nu$ by assumption,
there exists a set $\Gamma$ such that $\mu(\Gamma) = 0$ and $\nu(\Gamma^c) =
0$. Therefore, one must have $\psi(\Gamma) = \psi(\Gamma^c) = 0$ which
contradicts the fact that the measure $\psi$ is non-zero.
\end{proof}

\section{Stationary measure}\label{stationary}

\begin{proof}[Proof of Theorem \ref{j26.3} (ii)]

For a measure $\mu$ and function $f$, let $\mu(f)$ denote the integral of $f$ with respect to $\mu$.
Fix a continuous non-negative function $f: (\calT^d_1)^2 \to \R$ and note that $f$ is bounded, by compactness of $(\calT^d_1)^2$. 

Let $W$ denote Brownian motion on $(\calT^d_1)^2$ with the covariance matrix equal to  the unit diagonal matrix times $ 2 /( (d-1 ) d) $  and let $\E^w$ be the corresponding expectation, assuming that $W_0=w$. 
Standard coupling methods show that $W$ converges to the stationary distribution uniformly in $w$, that is, for every $\eps>0$ there exists $t_0$ such that for all $t\geq t_0$ and all $w \in (\calT^d_1)^2$, the Prokhorov distance between the distribution of $W_t$ and the uniform distribution on $(\calT^d_1)^2$ is less than $\eps$.

Fix an arbitrarily small $\eps_1>0$. 
By convergence of $W$ to the stationary distribution and the ergodic theorem, there exists $t_1$ so large that for any $w \in (\calT^d_1)^2$,
\begin{align}\label{jl29.1}
&\left|\E^{w} \left(\frac 1 {t_1} \int_0^{t_1} f(W_s) ds\right)
- (\nu^d_1 \times \nu^d_1)(f)\right|\\
&=
\left|\E^{w} \left(\frac 1 {t_1} \int_0^{t_1} f(W_s) ds\right)
- \lim_{t\to \infty} \frac 1 t \int_0^t f(W_s) ds\right| < \eps_1 /2. \nonumber
\end{align}

Fix $t_1$ satisfying the above estimate. Let
$\E^{z,x,y}$ denote the expectation corresponding to the distribution of $(B_t, X_t, Y_t)$ defined on $(\calT^d_r)^3$
starting from $(z,x,y)$. 
Recall from the last paragraph of Step 1 of the proof of Theorem \ref{j26.2} that the convergence of  $\{ n^{-1/2} (\fX^r_{ \sigma_{nt}} -\fX^r_0, \fY^r_{\sigma_{nt}} - \fY^r_0), t\in [0,t_1] \}$ to $\{W_t,  t\in [0,t_1] \}$ is uniform in the starting points of $X,Y$ and $B$.
It follows that there exists $r_1$ such that for all $r\geq r_1$, $z=B_0\in \calT^d_1$, and $(x,y)= w \in (\calT^d_1)^2$ such that  $|rx-rz|\geq 1$, $|rz-ry|\geq 1$ and $|rx-ry|\geq 2$, we have
\begin{align*}
\left|\E^{rz,rx,ry} \left( \frac 1 {t_1} \int_0^{t_1} f((X_{\sigma_{r^2 s}}, Y_{\sigma_{r^2 s}})/r) ds \right)
- \E^{w}\left( \frac 1 {t_1}  \int_0^{t_1} f(W_s) ds\right)\right| < \eps_1 /2.
\end{align*}
 The last estimate and \eqref{jl29.1} imply that
\begin{align*}
\left|\E^{rz,rx,ry} \left( \frac 1 {t_1} \int_0^{t_1} f((X_{\sigma_{r^2 s}}, Y_{\sigma_{r^2 s}})/r) ds \right)
- (\nu^d_1 \times \nu^d_1)(f)\right| < \eps_1 .
\end{align*}
By the Markov property applied at times $j t_1$, $j=1,2, \dots$, we obtain for any $k\geq 1$,
\begin{align} \label{jl29.2}
\left|\E^{rz,rx,ry} \left( \frac 1 {kt_1} \int_0^{kt_1} f((X_{\sigma_{r^2 s}}, Y_{\sigma_{r^2 s}})/r) ds \right)
- (\nu^d_1 \times \nu^d_1)(f)\right| < \eps_1 .
\end{align}
By Theorem \ref{j26.3} (i) and the ergodic theorem, the following limit exists a.s.,
\begin{align*}
\lim_{t\to \infty} \frac 1 t \int_0^t f((X_{\sigma_{r^2 s}}, Y_{\sigma_{r^2 s}})/r) ds
=
\lim_{k\to \infty}
\frac 1 {kt_1} \int_0^{kt_1} f((X_{\sigma_{r^2 s}}, Y_{\sigma_{r^2 s}})/r) ds ,
\end{align*}
so \eqref{jl29.2} and the Fatou lemma imply that 
\begin{align}\label{au12.1}
\E^{rz,rx,ry} \lim_{t\to \infty} \frac 1 t \int_0^t f((X_{\sigma_{r^2 s}}, Y_{\sigma_{r^2 s}})/r) ds \leq (\nu^d_1 \times \nu^d_1)(f) + \eps_1 .
\end{align}
Let $c_f = \sup _{x\in (\calT^r_1)^2} f(x)$. Then we can apply the same argument to the function $c_f - f(x)$ to see that 
\begin{align*}
\E^{rz,rx,ry} \lim_{t\to \infty} \frac 1 t \int_0^t (c_f- f((X_{\sigma_{r^2 s}}, Y_{\sigma_{r^2 s}})/r)) ds \leq (\nu^d_1 \times \nu^d_1)(c_f-f) + \eps_1 ,
\end{align*}
and, therefore,
\begin{align*}
\E^{rz,rx,ry} \lim_{t\to \infty} \frac 1 t \int_0^t f((X_{\sigma_{r^2 s}}, Y_{\sigma_{r^2 s}})/r) ds \geq (\nu^d_1 \times \nu^d_1)(f) - \eps_1 .
\end{align*}
We let $\eps_1$ go to 0 (and $r\to \infty$) in the last formula and \eqref{au12.1} to see that stationary distributions for the processes $(X_{\sigma_{r^2 t}}, Y_{\sigma_{r^2 t}})/r$ converge to
$\nu^d_1 \times \nu^d_1$.
This proves the theorem for processes run with the local time clock. We will show how this result implies the result for the processes run with the usual clock.

We will use results on excursion laws proved in Lemma \ref{au4.1}. There are two differences between the setup in that lemma and in the present proof. First, 
Lemma \ref{au4.1} contains estimates for lifetimes of excursion laws using the exit system for reflected Brownian motion in a domain with fixed holes. In the present context, the holes can move but this does not affect the validity of the estimates because the holes do not move during the lifetime of a single excursion. The second difference is that Lemma \ref{au4.1} is concerned with excursions of a single reflected Brownian motion. In the present context, the relevant Markov process is the vector $(B,X,Y)$ of which reflected Brownian motion is just one component. Hence, strictly speaking, we have to consider excursions of $(B,X,Y)$ from the set  $\{(b,x,y) \in (\calT^d_r)^3:
b \in \calS(x,1) \hbox{  or  } b \in \calS(y,1)\}$. It is easy to see that the estimates for the lifetime of an excursion derived in Lemma \ref{au4.1}
remain valid in the present context.

The exit system formula \eqref{exitsyst} and Lemma \ref{au4.1} (i) show that on the local time scale, the usual time is a jump process with the jump measure with finite expectation, bounded uniformly by a constant multiple of $r^d$. Part (ii) of Lemma \ref{au4.1} shows that outside a small set in the state space (small in the sense of having small Lebesgue measure relative to the measure of $\calT^d_r$), for sufficiently large $r$, the expectation of the jump measure is arbitrarily close to $r^d/s_d$. This implies that for any two subsets of the state space $(\calT^d_r)^3$, the ratio of the times spent by the process $(B,X,Y)$ in these sets in the long run will be the same as the ratio of local times spent by the process $(B,X,Y)$ in these sets. This observation and the fact that we have proved the theorem for the local time scale show that the theorem is true for the usual time scale.
\end{proof}

\section{Acknowledgments}

I am grateful to Jim Pitman and John Sylvester for very helpful advice.

\bibliographystyle{alpha}
\bibliography{twoboxes}

\end{document}